\documentclass[12pt]{amsart}

\usepackage{amsxtra,amssymb,amsmath,amscd,url,mathrsfs,tikz,fullpage,mathtools}
\usepackage[all]{xy}
\usepackage{tikz-cd}
\usepackage{color} 
\usepackage[initials]{amsrefs}
\usepackage[colorlinks]{hyperref}
\usepackage[colorinlistoftodos]{todonotes}
\usepackage{hyphenat}

\DeclareMathOperator{\disc}{disc}
\DeclareMathOperator{\res}{Res}
\DeclareMathOperator{\gal}{Gal}
\DeclareMathOperator{\rr}{r}

\newtheorem{theorem}{Theorem}[section] % 1st argument is your name for it
\newtheorem{lemma}[theorem]{Lemma}     % 2nd argument is what is printed
\newtheorem{corollary}[theorem]{Corollary}
\newtheorem{proposition}[theorem]{Proposition}
\newtheorem{conjecture}[theorem]{Conjecture}
\theoremstyle{definition}
\newtheorem{remark}{Remark}
\newtheorem{definition}{Definition}
\newtheorem{ex}{Example}

\renewcommand{\ge}{\geqslant}
\renewcommand{\leq}{\leqslant}
\renewcommand{\geq}{\geqslant}

\newcommand{\N}{\mathbf{N}}
\newcommand{\Z}{\mathbf{Z}}
\newcommand{\Q}{\mathbf{Q}}
\newcommand{\F}{\mathbf{F}}

\newcommand{\C}{\mathbf{C}}

\title{On the irreducibility and monodromy of Tutte polynomials}% This is the full title of the paper
\author{Andrew Goodall}
\address{Computer Science Institute (I\'UUK), Faculty of Mathematics and Physics, Charles University, Malostransk\'e n\'am. 25, 118 00 Praha 1, Czech Republic. {\tt andrew@iuuk.mff.cuni.cz}}
  
  \author{Florent Jouve}
  \address{Universit\'e de Bordeaux, CNRS, Bordeaux INP, IMB, UMR 5251, F-33400, Talence, France. {\tt florent.jouve@math.u-bordeaux.fr}}

\author{Jean-S\'ebastien Sereni}
\address{Service public français de la recherche, Centre National de la Recherche Scientifique (ICube, CSTB), Strasbourg, France. {\tt jean-sebastien.sereni@cnrs.fr}}

\makeatletter
\@namedef{subjclassname@2020}{\textup{2020} Mathematics Subject Classification}
\makeatother

\keywords{Tutte polynomials, matroids, Galois theory, monodromy, permutation groups}
\subjclass[2020]{05C31, 11R32, 11N36, 20B10}

\begin{document}

\begin{abstract}
We study algebraic properties of the Tutte polynomial of a matroid and its generalizations to other combinatorially defined bivariate polynomial invariants. Merino, de Mier and Noy showed that the Tutte polynomial of a connected matroid is irreducible, and Bohn, Cameron and M\"uller conjectured the stronger property that the Galois/mo\-nodromy group of the Tutte polynomial of a connected matroid of rank~$r$ is isomorphic to the full symmetric group on~$r$ letters. 
First, we generalize the result of Merino--de Mier--Noy to the context of general ranked sets by exploiting a recent translation of the Brylawski relations, satisfied by the coefficients of the Tutte polynomial, into a functional identity. Second, we give the first confirmation of the conjecture of Bohn--Cameron--M\"uller for infinite families of connected matroids, including the cycle graphs and the uniform matroids. Moreover, we apply the large sieve to obtain a probabilistic statement showing that suitable linear combinations of coprime Tutte polynomials generically satisfy the conjecture. 
\end{abstract}

\maketitle

\section{Introduction}\label{sec:intro}

The Tutte polynomial is a bivariate polynomial invariant of finite graphs that includes many important specializations, such as the
chromatic polynomial, reliability polynomial, partition function of the Potts model in statistical physics,
and the Jones polynomial of an alternating knot (encoded as a plane graph). The Tutte
polynomial is more generally defined as an invariant of finite matroids and in this guise has served to bridge disparate areas
of combinatorics (such as hyperplane arrangements and coding theory) and appears in various disciplines
(such as the random cluster model in physics and DNA sequencing in biology). Beyond matroids, the definition of the Tutte polynomial extends further to ranked sets~\cite{gordon15} where the rank function does not need satisfy all the axioms for a matroid rank function (see Section~\ref{sec:rk} below for a precise definition).

In the present work we continue the study initiated in~\cite{GJS} of ``generic'' properties of Tutte polynomials among bivariate~$\Z$-polynomials. While~\cite{GJS} focuses on linear algebraic aspects (what is the dimension of the linear span of natural finite families of Tutte polynomials?), here we seek to understand the expected algebro--geometric structure of such polynomials.

Let us recall that for~$M=(E,\rr)$ a matroid on groundset~$E$ and with rank function~$\rr$, its \emph{Tutte
polynomial} is defined by  
\[
T_M(x,y)=\sum_{A\subseteq E}(x-1)^{\rr(E)-\rr(A)}(y-1)^{|A|-\rr(A)}.
\]
It is an  element of~$\Z[x,y]$ with~$\deg_x T_M(x,y)=\rr(M)$  and~$\deg_y T_M(x,y)=|E|-\rr(M)$, where~$\rr(M)\coloneqq \rr(E)$. If~$M$ is the cycle matroid of a graph~$G=(V,E)$ with~$c(G)$ connected components then~$\rr(M)=|V|-c(G)$. The chromatic polynomial of~$G$ coincides (up to sign and a factor of~$x^{c(G)}$) with the specialization~$T_M(1-x,0)$.

  Given a matroid~$M=(E,\rr)$, its Tutte polynomial~$T_M(x,y)$ naturally defines a plane algebraic curve over~$\Q$. We will work in the slightly more general context where~$R$ is a fixed unique factorization domain (UFD) with field of fractions~$k$ 
  and where~$T_M(x,y)$ is seen as an element of~$R[x,y]$ through the application of the canonical ring homomorphism~$\phi\colon\Z\to R$ (sending~$1$ to~$1$) to its coefficients. In this framework we will denote by~$C_M/k$ the plane affine curve attached to~$T_M$, and by~$\mathbb A^1_k$ the affine line over~$k$. Our main focus in the present work is on properties of the~$k$-morphism~$\pi\colon C_M\to \mathbb A^1_k$ of degree~$\rr(M)$ defined for any extension~$K/k$ on~$K$-rational points~$(x,y)$ of~$C_M$ by~$\pi(x,y)=y$. 
  In the case where~$C_M/k$ is irreducible, two groups are naturally associated with the morphism~$\pi$ (see~\cite{Har79}*{\S 1}): 
 \begin{itemize}
 \item the Galois group of the (normal closure of the) field extension defined using the induced inclusion~$\pi^*\colon k(\mathbb A^1_k)=k(y)\to k(C_M)$,
 \item the monodromy group corresponding to the topological unbranched cover~$V\to U$ obtained by restricting~$\pi$ to suitable Zariski open subsets~$U$ and~$V$ of~$\mathbb A^1_k$ and~$C_M$, respectively. 
 \end{itemize}
 In the above setting Harris shows~(\cite{Har79}*{\S 1, Proposition}) that these two groups coincide, a fact that has been exploited for geometric purposes~(\cites{Har79,LeSo09,SoWh15}). In the sequel we will denote by~$G_{k,y}(T_M(x,y))$ (or~$G_{k,y}(T_M)$ for brevity) the Galois/monodromy group attached to~$T_M$ via the morphism~$\pi$ in the way described above. This group can be seen as a permutation group on~$\rr(M)$ letters and equals the Galois group of the splitting field of~$T_M(x,y)$ inside a fixed separable closure of~$k(y)$. The irreducibility assumption on~$C_M/k$ corresponds to the transitive action of~$G_{k,y}(T_M)$ seen as a permutation group, and this group is known to be a transitive subgroup of~$\mathfrak S_{\rr(M)}$ when~$M$ is a connected matroid~(\cite{MdMN}) as long as~$\mathrm{char}(k)$ does not divide the constant term of the monomial of degree~$1$ of the~$x$-polynomial~$T_M(x,y)$.
 One might wonder what finer properties than transitivity might be satisfied by the permutation group~$G_{k,y}(T_M)$, such as its action being primitive or doubly transitive, and whether these properties correspond to a structural property of~$M$ like transitivity does to connectivity. However, the following conjecture of Bohn--Cameron--M\"uller~\cite{BCM}*{Conj. 9} directs us rather to consider whether~$G_{k,y}(T_M)$
when transitive in fact coincides as a permutation group with the full symmetric group~$\mathfrak S_{\rr(M)}$:
 
 \begin{conjecture}[Bohn--Cameron--M\"uller]\label{conj:BCM}
Let~$M=(E,\rr)$ be a connected matroid on finite groundset~$E$ with~$\rr(M)>0$, and
let~$K$ be a field of characteristic zero. Then~$G_{K,y}(T_M)$ is maximal \emph{i.e} it is isomorphic to the symmetric group on~$\rr(M)$ letters.
\end{conjecture}

The conjecture has been computationally confirmed~\cite{BCM} for connected graphic matroids~$M$ such that~$\rr(M)\leq 10$.  
However, prior to the present work the conjecture has not as far as we know been proved for any infinite family of connected matroids. 
 In~\cite{BCM} 
 the analogue of Conjecture~\ref{conj:BCM} is proved for the multivariate Tutte polynomial of a matroid~(\cite{BCM}*{Th.~1} in which there are as many variables as elements in the groundset). On the other hand, specializing the Tutte polynomial at combinatorially meaningful values to obtain a univariate polynomial is not expected to lead to analogous behaviour. Indeed part of the motivation for stating and studying Conjecture~\ref{conj:BCM} comes from structural aspects of the chromatic polynomial. Once divided by the product of linear factors corresponding to its chromatic number, does the remaining ``interesting'' part of the chromatic polynomial of a graph resemble a random integer polynomial? For instance does it generically have maximal Galois group (as is the case for random integer polynomials~\cite{Bha25})?
It turns out that this does not seem to be the case (as theoretical and computational evidence obtained in~\cites{CaMo17,Mor12} indicate). In this sense, the bivariate Tutte polynomial of a matroid is currently believed to lie at the frontier of polynomial matroid invariants whose Galois group is almost surely maximal.

In another direction, to prove that the Galois group of the Tutte polynomial of a connected matroid is transitive we need only use the Brylawski relations satisfied by the coefficients and the fact that a matroid is connected if and only if the coefficient of~$x$ in its Tutte polynomial (Crapo's beta invariant~\cite{Crapo}) is non-zero. The class of bivariate polynomials satisfying Brylawski's relations includes not just Tutte polynomials of matroids but also rank-nullity polynomials of ranked sets generally. Drawing on a recent alternative derivation~\cite{BCCP23} of Brylawski's relations via the simplification the Tutte polynomial undergoes on the hyperbola~$(x-1)(y-1)=1$, we prove in Section~\ref{sec:irreducibility} that a large class of bivariate polynomials satisfying Brylawski's relations and with non-zero coefficient of~$x$ are irreducible, among which are the known cases of the Tutte polynomial of a connected matroid~\cite{MdMN} or connected delta matroid~\cite{EMGMNV22}. For some classes of ranked sets, the coefficient of~$x$ in the rank-nullity polynomial being non-zero is -- in a similar way to matroids -- a necessary and sufficient condition for not being a direct sum of smaller ranked sets (e.g. for delta matroids~\cite{EMGMNV22}), but for others the condition is only sufficient (e.g. for greedoids, see Remark~\ref{rem:exGordon} below).
 We also exhibit an infinite family of ranked sets whose Tutte polynomial is irreducible but does not have maximal Galois group. 

\medskip
The paper is organized as follows. In~\S\ref{sec:rk}, we study irreducibility properties of the Tutte
polynomial in the broader context of ranked sets. Inspired by the recent work of Beke et al.~\cite{BCCP23}, we
introduce the notion of a \emph{Brylawski polynomial},  which includes the rank-nullity polynomials of a
ranked set as a special case, and establish analogous criteria for the irreducibility of a Brylawski
polynomial to that established for the Tutte polynomial of a matroid~\cite{MdMN}. Next we prove
in~\S\ref{sec:prob} a probabilistic statement (which is, to a large extent, of independent interest) showing
that particular~$\Z[x]$-linear combinations of pairs of coprime bivariate polynomials that are monic of the
same degree seen as~$x$-polynomials (\emph{e.g.} the respective, and distinct, Tutte polynomials of two
connected matroids of the same size and rank) generically satisfy Conjecture~\ref{conj:BCM}. We do so by first
generalizing a sieve result of Gallagher towards van der Waerden's Conjecture. We conclude by showing
in~\S\ref{sec:fam} that Conjecture~\ref{conj:BCM} might be extended to some, but not all, more general ranked
sets and that it holds for some infinite families of connected matroids (including cycle graphs and
uniform~matroids).

\section{The Tutte polynomial of a ranked set}\label{sec:rk}

\begin{definition}\label{def:rank_function}
For a finite set~$E$, a \emph{rank function} on~$E$ is an integer-valued function~$\rr\colon 2^E\to\Z$
satisfying
\[\rr(\varnothing)=0,\quad
\rr(A)\leq \rr(E), \quad\mbox{ and } \quad \rr(A)\leq |A|,\quad \mbox { for all } A\subseteq E.\]
The pair~$S=(E,\rr)$ is called a \emph{ranked set}, and~$E$ is its \emph{groundset}.
\end{definition}

While~$0=\rr(\varnothing)\leq \rr(E)$, the rank function~$r$ may take negative values on proper subsets of~$E$, although when~$\rr$ is monotone, which is the case for matroids, antimatroids and greedoids, it takes nonnegative values on all subsets~\cite{Gordon12}.\footnote{A rank function~$\rr$ on~$E$ defines a \emph{matroid} if additionally it satisfies the following properties: 
\begin{itemize}
\item[(i)] Submodularity:  \[\rr(A\cup B)+\rr(A\cap B)\leq \rr(A)+\rr(B),\]   for~$A,B\subseteq E$; 
\item[(ii)] Adding an element cannot  decrease the rank (monotonicity), and any increase is by at most 1:
\[\rr(A)\leq \rr(A\cup\{e\})\leq \rr(A)+1,\]
for~$A\subseteq E$ and~$e\in E$.
\end{itemize}

For the rank function of a \emph{greedoid} the  additional conditions are weaker versions of those for a matroid:
\begin{itemize}
\item[(i)] Local submodularity:  If~$r(A)=r(A\cup \{e\})=r(A\cup\{f\})$, then~$r(A\cup\{e,f\})=r(A)$,   for~$A\subseteq E$ and~$e,f\in E$; 
\item[(ii)] Monotonicity:
\[r(A)\leq r(A\cup\{e\}),\]
for~$A\subseteq E$ and~$e\in E$.
\end{itemize}
Item (ii) implies that~$\rr$ takes nonnegative values. Antimatroids are special types of greedoids, the variation in item~(i) too complicated to state here. } 

To a ranked set~$S=(E,\rr)$ we associate the bivariate polynomial 
 \begin{equation}\label{eq:defTutte}
 T_S(x,y)=\sum_{A\subseteq E}(x-1)^{\rr(E)-\rr(A)}(y-1)^{|A|-\rr(A)},
 \end{equation}
which we call the corank-nullity polynomial of~$S$, better known as the Tutte polynomial when~$S$ is a matroid.
Along the hyperbola~$(x-1)(y-1)=1$ the polynomial~$T_S$ simplifies considerably to a polynomial dependent only on the size and rank of the groundset~$E$.    

\begin{proposition}\label{prop:TutteBrylawski}
For a ranked set~$S=(E,\rr)$,
\begin{equation}\label{eq:hyperbola1}
 (y-1)^{\rr(E)}T_S\big(\frac{y}{y-1},y\big)=y^{|E|}.
\end{equation}
In other words, 
\[T_S(x,y)\equiv x^{\rr(E)}y^{|E|-\rr(E)}\quad\pmod{xy-x-y}.\]
\end{proposition}
\begin{proof}
 By the definition of~$T_S$, the left-hand side of~\eqref{eq:hyperbola1} equals
\[
 \sum_{A\subseteq E}\Big(\frac{y-(y-1)}{y-1}\Big)^{\rr(E)-\rr(A)}(y-1)^{\rr(E)+|A|-\rr(A)} =
  \sum_{A\subseteq E}(y-1)^{|A|}=y^{|E|}.
\]
\end{proof}

Suppose~$S_1=(E_1, \rr_1)$ and~$S_2=(E_2, \rr_2)$ in which~$\rr_1$ and~$\rr_2$ are rank functions on disjoint sets~$E_1$ and~$E_2$, respectively. Then the \emph{direct sum}~$S_1\oplus S_2=(E_1\cup E_2, \rr)$ has rank function on~$E=E_1\cup E_2$ defined by~$\rr(A_1\cup A_2)=\rr_1(A_1)+\rr_2(A_2)$, where~$A_1\subseteq E_1$ and~$A_2\subseteq E_2$. Then
$T_{S_1\oplus S_2} (x,y)=T_{S_1}(x,y)T_{S_2}(x,y)$.
In other words, if~$S$ can be expressed as the direct sum of ranked sets on non-empty groundsets, then ~$T_S(x,y)$ is a reducible polynomial. 
When~$S$ is a matroid, Merino et al.~\cite{MdMN} showed that if~$T_S(x,y)$ is reducible, then there are non-empty matroids~$S_1$ and~$S_2$ such that~$S=S_1\oplus S_2$, i.e.~$S$ is not connected.
They use the fact that when~$|E|\geq 2$ the matroid~$S$ is connected if and only if the coefficient of~$x$ in~$T_S(x,y)$ is non-zero (a result due to Crapo~\cite{Crapo}).\footnote{The proof of this depends on the fact that for a connected matroid~$S$ and~$e\in E$
either~$S\backslash e$ or~$S/e$ is also connected~\cite{Tutte}*{\S 6.5} (see also~\cite{Crapo}*{p.~410}); that~$t_{1,0}$, like all the coefficients of the Tutte polynomial, is non-negative (as is easily seen by induction using its deletion-contraction recurrence); and finally that~$t_{1,0}$ satisfies the deletion-contraction recurrence for the Tutte polynomial.  These properties do not extend to ranked sets generally.}
For greedoids (ranked sets generally) there are examples of~$S$ not expressible as the direct sum of greedoids (ranked sets) on non-empty groundsets such that~$T_S(x,y)$ is reducible, and examples of~$S$ such that~$T_S(x,y)$ is irreducible but the coefficient of~$x$ is zero. (For the latter, see Remark~\ref{rem:exGordon} below.)       
\begin{definition}\label{def:dual}
Let~$S=(E,\rr)$ where~$E$ is a finite set and~$\rr$ is a rank function on~$E$.
The dual rank function~$\rr^*$ is defined by
\[\rr^*(A)=|A|+\rr(E\setminus A)-\rr(E),\]
for~$A\subseteq E$, and we write~$S^*=(E,\rr^*)$ for the corresponding ranked set. 
\end{definition}
The dual~$\rr^*$ of a rank function is again a rank function because ~$\rr^*(\varnothing)=0=\rr(\varnothing)$; 
the condition~$\rr^*(A)\leq |A|$ is satisfied for all~$A\subseteq E$ as
$\rr^*(A)=|A|-[\rr(E)-\rr(E\setminus A)]\leq |A|$ since~$\rr(E\setminus A)\leq\rr(E)$;
and
$\rr^*(A)=|A|+\rr(E\setminus A)-\rr(E)\leq |A|+|E\setminus A|-\rr(E)=|E|-r(E)=\rr^*(E)$.
The dual of the dual rank function satisfies
$(\rr^*)^*(A)=|A|+|E\setminus A|+\rr(A)-\rr(E)-[|E|-\rr(E)] = \rr(A)$.
Thus~$(S^*)^*=S$, and it is routine to verify the duality formula~$T_{S^*}(x,y)=T_S(y,x)$. 

We have~$\deg_x T_S=\rr(E)-\min_A \rr(A)\geq \rr(E)$ with equality if and only if~$\rr$ takes just non-negative values;
and, noting that~$|A|-\rr(A)=\rr^*(E)-\rr^*(E\setminus A)$,~$\deg_y T_S=\rr^*(E)-\min_A\rr^*(A)\geq |E|-\rr(E)$, with equality if and only if~$\rr^*$ takes just non-negative values.

\subsection{Brylawski polynomials}

\begin{definition}\label{def:BryPol}
A bivariate polynomial~$U(x,y)$  over a commutative ring~$R$ such that
~$(y-1)^rU(\frac{y}{y-1},y)=cy^n$ for some non-negative integer~$n$, integer~$r$, and non-zero constant~$c\in R$ is 
called an~$(n,r)$-\emph{Brylawski polynomial} (with constant~$c$).
Equivalently,~$(x-1)^{n-r}U(x,\frac{x}{x-1})=cx^n$.
\end{definition}

In Definition~\ref{def:BryPol}, we may have~$r<0$: for example,~$U(x,y)=(y-1)^{\ell}$ is a ~$(0,-\ell)$-Brylawski polynomial (with constant~$c=1$). Similarly, we may have~$r>n$: for example,~$(x-1)^k$ is a~$(0,k)$-Brylawski polynomial (with constant~$c=1$). 
By Proposition~\ref{prop:TutteBrylawski}, the corank-nullity polynomial~$T_S(x,y)$ of a ranked set~$S=(E,\rr)$ is an~$(|E|, \rr(E))$-Brylawski polynomial (with constant~$c=1$).
Here~$0\leq r=\rr(E)\leq n=|E|$.
 
More generally, polynomials of the form~$c_0x^{k_1}y^{\ell_1}(x-1)^{k_2}(y-1)^{\ell_2}$, with~$c_0\in R\setminus\{0\}$,~$k_i,\ell_j\in \Z_{\geq 0}$, are clear instances of Brylawski polynomials (for the choice~$(n,r)=(k_1+\ell_1,k_1+k_2-\ell_2)$).

If~$U(x,y)$ is an~$(n,r)$-Brylawski polynomial and~$V(x,y)$ is an~$(m,s)$-Brylawski polynomial, then~$U(x,y)V(x,y)$ is an~$(n+m,r+s)$-Brylawski polynomial. (Lemma~\ref{lem:product_Brylawski} below gives the converse.) In particular, if~$U(x,y)$ is an~$(n,r)$-Brylawski polynomial then~$x^{k_1}y^{\ell_1}(x-1)^{k_2}(y-1)^{\ell_2}U(x,y)$ is an~$(n+k_1+\ell_1, r+k_1+k_2-\ell_2)$-Brylawski polynomial. 

For a polynomial~$U(x,y)=\sum u_{i,j}x^iy^j\in R[x,y]$, we set~$\deg_x U=\max\{i:\exists j\:\: u_{i,j}\neq 0\}$
and~$\deg_y U=\max\{j:\exists i\:\: u_{i,j}\neq 0\}$.

\begin{proposition}\label{prop:degree_bounds}
If~$U(x,y)\in R[x,y]$ is an~$(n,r)$-Brylawski polynomial, then~$\deg_x U\geq r$ and~$\deg_y U\geq n-r$.

Moreover, if~$R$ is a UFD, then any univariate factors of~$U(x,y)$ are necessarily of the form~$a(x-1)^kx^\ell$ or~$b(y-1)^ky^\ell$, for~$a,b\in R\setminus\{0\}$ and~$k,\ell\in \Z_{\geq 0}$.
\end{proposition}
\begin{proof}
Let~$U(x,y)=\sum u_{i,j}x^iy^j$ and~$d\coloneqq \deg_x$. 
Since~$\sum_{i,j} u_{i,j}(y-1)^{d-i}y^{i+j}=c(y-1)^{d-r}y^n$ is a polynomial identity, the degree~$d$
satisfies~$d\geq r$. Moreover, if~$d>r$, then, by setting~$y=1$ in this identity, ~$\sum_{j}u_{d,j}=0$, while if~$d=r$ then~$\sum_{j}u_{d,j}=c$.   
Likewise, as the identity~$(y-1)^rU(\frac{y}{y-1},y)=cy^n$ is equivalent to the identity~$(x-1)^{n-r}U(x,\frac{x}{x-1})=cx^n$, a similar argument shows that~$\deg_y U \geq n-r$.

We turn to the second statement. If~$U(x,y)$ is a Brylawski polynomial divisible by a polynomial~$V(x)$ (independent of~$y$), then there exists~$U_0(x,y)\in R[x,y]$ such that
\[
(x-1)^{n-r}U(x,\tfrac x{x-1})=V(x)(x-1)^{n-r}U_0(x,\tfrac x{x-1})=cx^n.
\]
By unique factorization in~$R[x]$ the only potential irreducible factors of~$V(x)$ are~$x$ and~$x-1$.
We reach the analogous conclusion for factors of~$U(x,y)$ of type~$V(y)\in R[y]$ by using the identity~$(y-1)^rU(\frac{y}{y-1},y)=cy^n$.
\end{proof}

\begin{remark} We can say a little more in Proposition~\ref{prop:degree_bounds} when~$U(x,y)$ is the corank-nullity polynomial~$T_S(x,y)$ of a ranked set~$S=(E,\rr)$.  
As already mentioned,~$T_S$ is an~$(|E|,\rr(E))$-Brylawski polynomial (by Proposition~\ref{prop:TutteBrylawski}), and we have equality in~$\deg_x T_S\geq \rr(E)$ if and only if~$\rr(A)\geq 0$ for each~$A\subseteq E$, and equality in~$\deg_y T_S\geq |E|-\rr(E)$ if and only if~$\rr^*(A)\geq 0$ for each~$A\subseteq E$.

Moreover the only possible univariate factors of~$T_S(x,y)$ are of the form~$ax^\ell$ and~$by^\ell$. For suppose~$x-1$ divides~$T_S(x,y)$. Then 
 \[T_S(1,y)=\sum_{\stackrel{A\subseteq E}{\rr(A)=\rr(E)}}(y-1)^{|A|-\rr(E)}=0,\]
 which cannot hold as the sum contains leading term~$(y-1)^{|E|-\rr(E)}$, and so is monic as a polynomial in~$y$. 
 Dually, supposing~$y-1$ divides~$T_S(x,y)$ yields  
 \[T_S(x,1)=\sum_{\stackrel{A\subseteq E}{\rr(A)=|A|}}(x-1)^{\rr(E)-|A|}=0,\]
 and the sum is a monic polynomial in~$x$ (from the term contributed by~$A=\varnothing$). 
 
 When~$S=(E,\rr)$ is a matroid,  if~$x^\ell$ divides~$T_S(x,y)$ then ~$S$ is the direct sum of a smaller ranked set and~$S_1=(E_1,\rr_1)$, in which~$|E_1|=\ell$ and ~$\rr_1(A)=|A|$ for each~$A\subseteq E_1$ (i.e.,~$S_1$ is the matroid of~$\ell$ coloops, for which~$T_{S_1}(x,y)=x^\ell$); dually, 
 if~$y^\ell$ divides~$T_S(x,y)$ then~$S$ is the direct sum of a smaller ranked set and ~$S_1^*=(E_1,\rr_1^*)$, in which~$|E_1|=\ell$ and~$\rr_1^*(A)=0$ for each~$A\subseteq E_1$ (i.e.,~$S_1^*$ is the matroid of~$\ell$ loops, for which~$T_{S_1^*}(x,y)=y^\ell$).  
 
 \end{remark} 
The motivation for the name ``Brylawski polynomial" lies in the following key proposition and its consequence for corank-nullity polynomials of ranked sets (Corollary~\ref{cor:Brylawski}). 

\begin{proposition}[\cite{BCCP23}*{proof of Theorem 1.1}]\label{prop:Brylawski}
If~$U(x,y)$ is an~$(n,r)$-Brylawski polynomial with constant~$c$ 
then its coefficients~$u_{i,j}$ satisfy \footnote{We use the convention~$\binom{a}{b}=0$ if~$b<0$ or~$b>a$.}
\[\sum_{\stackrel{i,j}{i+j\leq h}} (-1)^{j}\binom{h-i}{j}u_{i,j}=c\cdot(-1)^{n-r}\binom{h-r}{h-n}, \quad\mbox{for any integer~$h\geq 0$}.\]
\end{proposition}
In particular,~$u_{0,0}=0$ if~$n>0$, and~$u_{1,0}=u_{0,1}$ if~$n>1$. 
Proposition~\ref{prop:Brylawski} contains the following special case.
\begin{corollary}\label{cor:Brylawski}
Let~$S=(E,\rr)$ be a ranked set with~$|E|=n$,~$\rr(E)=r$, and  
$T_S(x,y)=\sum_{i,j}t_{i,j}x^iy^j$. 
Then, for any integer~$h\geq 0$,
\begin{equation}\label{Brylawski}\sum_{\stackrel{i,j}{i+j\leq h}} (-1)^{j}\binom{h-i}{j}t_{i,j}=(-1)^{n-r}\binom{h-r}{h-n}.\end{equation}
\end{corollary}
The linear relations given by Corollary~\ref{cor:Brylawski} for~$0\leq h<n$ were established by Brylawski~\cite{B72} for matroid rank functions, and extended to greedoid and antimatroid rank functions by Gordon~\cite{gordon15}, who also established the affine relation for~$h=n$. 
Although the proof of Corollary~\ref{cor:Brylawski} given by Beke\emph{~et al.}~\cite{BCCP23} assumes that~$\rr$ takes non-negative values, it is easily extended to include the case of rank functions taking negative values as well. 
In fact~\cite{BCCP23}*{Th.~1.1} is really Corollary~\ref{cor:Brylawski} above once extended from~$\N$- to~$\Z$-valued rank functions: 
Beke\emph{~et al.} do not explicitly state the generalization to Brylawski polynomials, although their argument depends only on the property defining a Brylawski polynomial (Definition~\ref{def:BryPol}) and not on being the Tutte polynomial of a rank function. 

A constant not equal to 1 in Definition~\ref{def:BryPol} of an~$(n,r)$-Brylawski polynomial features in~\cite{BCCP24}*{Lemma 7.6, Theorem 8.2}. For our purposes, allowing an arbitrary constant serves as a technical convenience (and we do not usually need to specify the constant) enabling one to prove stability properties of the class of Brylawski polynomials. While it is straightforward to see that the product of two Brylawski polynomials is again a Brylawski polynomial, the following lemma asserts that the converse holds over any UFD. 

\begin{lemma}\label{lem:product_Brylawski} Let~$R$ be a UFD. 
Suppose that~$T(x,y)$ is an~$(n,r)$-Brylawski polynomial in~$R[x,y]$
 with factorization ~$T(x,y)=U(x,y)V(x,y)$ in~$R[x,y]$.
Then there are integers~$m,s$ such that~$U(x,y)$ is an~$(n-m,r-s)$-Brylawski polynomial and~$V(x,y)$ is an~$(m,s)$-Brylawski polynomial, where~$0\leq m\leq n$.
\end{lemma}
\begin{proof}
For simplicity write~$T(x,y)=T(x)$, meaning that it is considered as a polynomial in~$x$ with coefficients in~$R[y]$. Likewise, write
 the given
factorization of~$T$ over~$R[y]$ as~$T(x)=U(x)V(x)$.
As~$T$ is an~$(n,r)$-Brylawski polynomial,
there is non-zero~$c\in R$ such that
\begin{equation}\label{eq:product}cy^n=(y-1)^rT\big(\frac{y}{y-1}\big)=(y-1)^{r-s}U\big(\frac{y}{y-1}\big)\cdot (y-1)^sV\big(\frac{y}{y-1}\big),\end{equation}
where~$s$ is chosen to be the minimum integer such that~$(y-1)^sV(\frac{y}{y-1})$ is a polynomial in~$y$. This then forces~$(y-1)^{r-s}U(\frac{y}{y-1})$ to be a polynomial in~$y$, for otherwise it is a ratio of a polynomial in~$y$ and a power of~$y-1$, and this implies that~$y-1$ divides~$(y-1)^{s}V(\frac{y}{y-1})$ as the product on the right-hand side of~\eqref{eq:product} is the polynomial~$cy^n$ on the left-hand side, which contradicts minimality of~$s$.

By unique factorization in~$R[y]$, this implies that
\[(y-1)^sV\big(\frac{y}{y-1}\big)=b y^m, \quad\mbox{ and }(y-1)^{r-s}U\big(\frac{y}{y-1}\big)=a y^{n-m},\]
for some~$a,b\in R$ such that~$ab=c$ and some integer~$m$ satisfying~$0\leq m\leq n$.
\end{proof}

\subsection{Irreducibility of Brylawski polynomials}\label{sec:irreducibility}
The main result of this section is the following.
\begin{theorem}\label{th:MdMN}
Let~$R$ be a UFD and let ~$T(x,y)=\sum_{i,j}t_{i,j}x^iy^j\in R[x,y]$ be an~$(n,r)$-Brylawski polynomial, where~$n,r$ are integers with~$r\geq 1$ and~$n-r\geq 1$. Suppose %hence~$t_{0,0}=0$
 that {\rm (i)} neither~$x-1$ nor~$y-1$ divide~$T(x, y)$; 
 {\rm (ii)}~$t_{1,0}\neq 0$ and {\rm (iii)}~$\deg_x T+\deg_y T\in\{n, n+1\}$.
If~$T(x,y)=U(x,y)V(x,y)$ in~$R[x,y]$ then~$U(x,y)$ or~$V(x,y)$ is constant (in~$R$). In particular~$T(x,y)$ is irreducible in~$k[x,y]$ where~$k$ is the fraction field of~$R$; moreover if~$\mathrm{gcd}_{i,j}\{t_{i,j}\}=1$ then~$T(x,y)$ is irreducible in~$R[x,y]$.
\end{theorem}
\noindent 
By Proposition~\ref{prop:degree_bounds}, condition (iii) is equivalent to~$\deg_x T+\deg_y T\leq n+1$.

\begin{proof}[Proof of Theorem~\ref{th:MdMN}]
Suppose that~$T(x,y)$ is an~$(n,r)$-Brylawski polynomial and~$T(x,y)=U(x,y)V(x,y)$, where~$U, V$ are not constants (in~$R$).  By assumption~(i), and the fact that under assumption~(ii)
neither~$x$ nor~$y$ divide~$T(x,y)$ (noting that~$t_{0,1}=t_{1,0}$ by Proposition~\ref{prop:Brylawski} as~$n\geq 2$), Proposition~\ref{prop:degree_bounds} implies neither~$U(x,y)$ nor~$V(x,y)$ is univariate. By Lemma~\ref{lem:product_Brylawski} there are integers~$m,s$ such that~$U(x,y)$ is an~$(n-m,r-s)$-Brylawski polynomial and~$V(x,y)$ is an~$(m,s)$-Brylawski polynomial, where~$0\leq m\leq n$.
We also have~$\deg_x U+\deg_x V = \deg_x T$ and~$\deg_y U + \deg_y V= \deg_y T$.  
By definition, there is~$c\in R$ such that  
\[(y-1)^r T(\frac{y}{y-1},y)=cy^n.\]
We may assume
that~$0<\deg_x U, \deg_x V<\deg_x T$ and~$0<\deg_y U, \deg_y V<\deg_y T$. For otherwise~$T(x,y)$ admits a univariate factor in~$R[x,y]$, and we reach a contradiction. 

As~$U(x,y)$ is an~$(n-m, r-s)$-Brylawski polynomial and~$V(x,y)$ is a~$(m,s)$-Brylawski polynomial 
where~$0\leq m\leq n$,
by Proposition~\ref{prop:degree_bounds} we have
$s\leq \deg_x V<\deg_x T$ and~$m-s\leq\deg_y V<\deg_y T$. These inequalities imply that~$m\leq \deg_x T+\deg_y T-2\leq n-1$.
Similarly,  
$r-s\leq \deg_x U<\deg_x T$ and~$n-m-(r-s)\leq \deg_y U<\deg_y T$
imply that~$n-m\leq \deg_x T+\deg_y T-2\leq n-1$.
Hence~$m\geq 1$ and~$n-m\geq 1$, and the first of Brylawski's relations gives~$v_{0,0}=0=u_{0,0}$.  

 But then~$t_{1,0}=u_{0,0}v_{1,0}+u_{1,0}v_{0,0}=0$, contrary to assumption~(ii). 
\end{proof}

When~$T=T_S$ is the corank-nullity polynomial of a ranked set~$S=(E,\rr)$ in which both~$\rr$ and its dual~$\rr^*$ take nonnegative values (which holds for rank functions of matroids, but not necessarily of greedoids~\cite{Gordon12}), 
we have~$\deg_x T_S=\rr (E)$ and~$\deg_y T_S=|E|-\rr(E)$ and so, with~$\deg_x T_S+\deg_y T_S=|E|$, condition~(iii) in Theorem~\ref{th:MdMN} holds for the~$(|E|,\rr(E))$-Brylawski polynomial~$T_S(x,y)$. But if either~$\rr$ or~$\rr^*$ take negative values, then condition~(iii) fails for~$T_S$ unless only one of these rank functions takes negative values and moreover the only negative value this one takes is~$-1$.
We therefore deduce the following particular case of Theorem~\ref{th:MdMN}, first shown by Merino, de Mier and Noy~\cite{MdMN}.

\begin{corollary}\label{cor:dMMN}
Let~$M$ be a connected matroid on groundset~$E$ of size~$\geq 2$ and with Tutte polynomial~$T_M(x,y)=\sum_{i,j}t_{i,j}x^iy^j$. Let~$R$ be any UFD and let~$\varphi\colon \Z\to R$ be the natural ring homomorphism whose kernel is generated by the characteristic~$\kappa$ of~$R$. Assuming that~$\kappa\nmid t_{1,0}$, then~$T_M$, seen as an element of~$R[x,y]$ by applying~$\varphi$ to its coefficients, is irreducible.
\end{corollary}
For a connected matroid~$S=(E,\rr)$ with~$|E|\geq 2$, we have~$t_{1,0}=t_{0,1}\neq 0$, ensuring that~$r=\rr(E)\geq 1$ and~$n-r=|E|-\rr(E)\geq 1$ as required for Theorem~\ref{th:MdMN}.

Corollary~\ref{cor:dMMN} extends to delta matroids in which the rank function~$\rr$ used in the definition of the Tutte polynomial is the function~$\sigma$ defined in~\cite{EMGMNV22}*{p.~1341, eq.~(3)} (obtained by averaging the ranks of two associated matroids).
By Theorem~4.7 of the paper just cited, the coefficient~$t_{1,0}$ is non-zero if and only if the delta matroid is connected; that paper's Theorem~1.2 is then a consequence of the general phenomenon recorded in our Theorem~\ref{th:MdMN}.     

\begin{remark}\label{rem:exGordon}
Assumptions~(ii) and~(iii) in Theorem~\ref{th:MdMN} are not necessary conditions for irreducibility. (As already mentioned, for matroids  assumption~(ii) is in fact necessary as well as sufficient for connectivity, and accordingly necessary and sufficient for irreducibility~\cite{MdMN}.)

The polynomial~$x^3+2x^2+y^2+3xy$ (see~\cite{gordon15}*{p.~23}) is a~$(5,3)$-Brylawski polynomial, equal to the corank-nullity polynomial  \[T_S(x,y)=(x-1)^3+5(x-1)^2+10(x-1)+7+3(x-1)(y-1)+5(y-1)+(y-1)^2,\]
where~$S=(E,\rr)$, in which~$|E|=5$ and~$\rr$ is defined by
$\rr(\varnothing)=0$;~$\rr(A)=1$ for each~$A$ of size 1;~$\rr(A)=2$ for each~$A$ of size 2;~$\rr(A)=3$ for 7 subsets~$A$ of size 3, and~$\rr(A)=2$ for the remaining 3 subsets of size 3;~$\rr(A)=3$ for each~$A$ of size 4; and~$\rr(E)=3$.
The polynomial~$T_S(x,y)$ shares with the Tutte polynomial of a  matroid the properties of being monic as a polynomial in~$x$ (considering~$T_S(x,y)$ over~$\Z[y]$) and monic as a polynomial in~$y$ (considering~$T_S(x,y)$ over~$\Z[x]$).
This polynomial is irreducible (as a quadratic in~$y$, its discriminant is~$x^2(1-4x)$) but does not satisfy assumption~(ii) of Theorem~\ref{th:MdMN}.
Irreducibility of~$T_S(x,y)$   implies the ranked set~$S$ cannot be expressed as the direct sum of smaller ranked~sets.\footnote{Gordon~\cite{Gordon97} exhibits in his Example~2.1 a greedoid that is not decomposable as a direct sum of smaller greedoids and yet has~$t_{1,0}=0$.
Let~$E=\{a,b,c\}$ and~$\mathcal{F}=\{\varnothing, \{a\}, \{b\}, \{a,c\}, \{b,c\}, \{a,b,c\}\}$. The greedoid rank function is defined for~$A\subseteq E$ by \[\rr(A)=\max\{|F|:F\in\mathcal F, F\subseteq A\}.\]
Then~$S=(E,\rr)$ has corank-nullity polynomial
\begin{align*}T_S(x,y)&= (x\!-\!1)^3+2(x\!-\!1)^2+(x\!-\!1)^3(y\!-\!1)+2(x\!-\!1)+(x\!-\!1)^2(y\!-\!1)+\!1\\
& = x[x+(x-1)^2y]\end{align*}
The first factor is the Tutte polynomial of a single isthmus, i.e.\@ of the ranked set~$S_1=(E_1,\rr_1)$, where~$E_1=\{c\}$ and~$\rr_1$ is defined by 
$\rr_1(\varnothing)=0, \rr(\{c\})=1$. 
While the second factor is not the Tutte polynomial of a greedoid, it is equal to 
\[(x-1)^2(y-1)+(x-1)^2+x-1+1\]
which is the corank-nullity polynomial of~$S_2=(E_2, \rr_2)$ in which~$E_2=\{a,b\}$ and~$\rr_2$ is defined by 
\[\rr_2(\varnothing)=0, \: \rr_2(\{a\})=0,\: \rr_2(\{b\})=1, \: \rr_2(\{a,b\})=2.\]
Not only do we have then 
\[T_S(x,y)=T_{S_1}(x,y)T_{S_2}(x,y),\]
but~$S=S_1\oplus S_2$ even though as a greedoid~$S$ cannot be expressed as a sum of smaller greedoids.}

Assumption~(iii) --  while, as already noted, automatic for corank-nullity polynomials of matroids -- is not necessary either, as the following example shows.  
We define the ranked set~$S=(E,\rr)$ for integers~$a>b>0$, in which~$|E|\geq b$, by setting
\[\rr(A)=\begin{cases} 0 & A=\varnothing,\\
|A|-a & \varnothing\subsetneq A\subsetneq E,\\
|E|-b & A= E.\end{cases}\]
Then 
\[T_S(x,y)=(x-1)^{|E|-b}+\sum_{\varnothing\subsetneq A\subsetneq E}(x-1)^{|E|-b-|A|+a}(y-1)^{|A|-|A|+a}+(y-1)^{|E|-|E|+b}\]
\[=(x-1)^{|E|-b}+(x-1)^{a-b}(y-1)^a\sum_{0<i<|E|}\binom{|E|}{i}(x-1)^i+(y-1)^b.\]
Also,
\[\rr^*(A)=\begin{cases} 0 & A=\varnothing,\\
b-a & \varnothing\subsetneq A\subsetneq E,\\
b & A= E.\end{cases}\]
Writing~$X=x-1$,~$Y=y-1$,~$|E|=n$,
\[T_S(X,Y)=X^{n-b}+X^{a-b}Y^a\sum_{0<i<n}\binom{n}{i}X^i + Y^b.\]
This has~$X$-degree~$n+(a-b-1)\geq n\geq r(E)=n-b$ and~$Y$-degree~$a$, so that~$\deg_x T_S+\deg_y T_S =n+2a-b-1\geq n+a$.
Thus for~$a>1$ the assumption~(iii) fails.  The Newton polygon of~$T_S(X,Y)$ has vertices~$(0,b)$,~$(n-b,0)$, and~$(a-b+i,a)$ for~$i\in\{n-1,\dotsc, 1\}$, and is readily seen not to be reducible as a Minkowski sum of smaller polygons. (Its sides have direction vectors~$(n-b,-b)$,~$(a-1,a)$,~$(-1,0)$ ($n-2$ times), and~$(b-a-1,b-a)$; since $a,b, a-b>0$, all the sides of one of these smaller polygons would be forced to have direction vectors~$(-1,0)$.) Hence~$T_S(X,Y)$ is irreducible~(\cite{Gao01}*{p.~507}).  
\end{remark}

The example in Remark~\ref{rem:exGordon} in which assumption~(ii) of Theorem~\ref{th:MdMN} fails is accommodated by the following irreducibility criterion, derived by a similar proof to that of  Theorem~\ref{th:MdMN}. 

 \begin{theorem}\label{th:irreducible_2}
Let~$T(x,y)=\sum_{i,j}t_{i,j}x^iy^j$ be an~$(n,r)$-Brylawski polynomial over~$R$, a UFD, where~$n,r$ are integers with~$r\geq 1$ and~$n-r\geq 1$. Let~$\varphi\colon \Z\to R$ be the canonical ring homomorphism and suppose
 that {\rm (i)} none of~$x-1, y-1, x, y$ divide~$T(x, y)$; {\rm (ii)}~$\varphi(2)\nmid t_{1,1}$; and {\rm (iii)}~$\deg_x T+\deg_y T=n$.
Then~$T(x,y)$ is irreducible in~$k[x,y]$ where~$k$ denotes the fraction field of~$R$; moreover if~$\mathrm{gcd}_{i,j}\{t_{i,j}\}=1$ in~$R$ then~$T(x,y)$ is irreducible in~$R[x,y]$.
\end{theorem}

\begin{proof}
Suppose that~$T(x,y)=U(x,y)V(x,y)$ is factorization of~$T(x,y)$ in~$R[x,y]$ with both~$U$ and~$V$ non constant. Assumption (iii) (together with Proposition~\ref{prop:degree_bounds}) implies~$\deg_x U+\deg_x V = \deg_x T=r$ and~$\deg_y U + \deg_y V= \deg_y T=n-r$, and  
as~$T(x,y)$ is an~$(n,r)$-Brylawski polynomial, there is~$c\in R$ such that  
\[(y-1)^r T(\frac{y}{y-1},y)=cy^n.\]
As in the proof of Theorem~\ref{th:MdMN}, we may assume
that~$0<\deg_x U, \deg_x V<\deg_x T$ and~$0<\deg_y U, \deg_y V<\deg_y T$ (needing (i); note that assuming that neither~$x$ nor~$y$ are factors of~$T(x,y)$ replaces the condition that~$t_{1,0}\neq 0$ assumed in the earlier theorem), and, using Lemma~\ref{lem:product_Brylawski}, the polynomial~$U(x,y)$ is an~$(n-m, r-s)$-Brylawski polynomial and~$V(x,y)$ is a~$(m,s)$-Brylawski polynomial
for some integers~$m, s$ with~$0\leq m\leq n$.
We invoke Proposition~\ref{prop:degree_bounds} and see that the inequalities
$s\leq \deg_x V<\deg_x T$ and~$m-s\leq\deg_y V<\deg_y T$
imply that~$m\leq \deg_x T+\deg_y T-2= n-2$;
and  
$r-s\leq \deg_x U<\deg_x T$ and~$n-m-(r-s)\leq \deg_y U<\deg_y T$
imply~$n-m\leq \deg_x T+\deg_y T-2=n-2$.
Hence~$m\geq 2$ and~$n-m\geq 2$, and the first two Brylawski relations (Proposition~\ref{prop:Brylawski}) give~$v_{0,0}=0=u_{0,0}$ and~$u_{1,0}=u_{0,1}$ and~$v_{1,0}=v_{0,1}$.  

 But then~$t_{1,1}=u_{0,0}v_{1,1}+u_{1,1}v_{0,0}+u_{1,0}v_{0,1}+u_{0,1}v_{1,0}=\varphi(2)u_{0,1}v_{0,1}$, contradicting (ii). 
\end{proof}

We close this section with two examples of ranked sets for which our general criteria for irreducibility in
Theorems~\ref{th:MdMN} and~\ref{th:irreducible_2} are not satisfied, but which we can prove are irreducible by
considering their Newton polygons (again applying~\cite{Gao01}*{p.~507}). Each example involves a rank
function taking non-negative values, but whose dual takes negative values. In Section~\ref{sec:ex1} we show
that the Galois group of the corank-nullity polynomial in Example~\ref{ex:1_nonneg_dual_neg} as a polynomial
in~$x$ is not the full symmetric group of degree~$r$ (see Proposition~\ref{prop:non-max}; this is the only
example we know where Conjecture~\ref{conj:BCM} does not extend from matroids to ranked sets generally). In Section~\ref{sec:ex2} we
exhibit an infinite number of instances of the corank-nullity polynomial in Example~\ref{ex:2_nonneg_dual_neg}
for which its Galois group as a polynomial in~$x$ is the symmetric group of degree~$|E|$ (see
Corollary~\ref{cor:pos-proportion}).  

\begin{ex}\label{ex:1_nonneg_dual_neg} Let~$E$ be a finite set and let~$r\in\Z$ be such that~$|E|\geq r\geq 1$. Consider~$\rr\colon 2^E\to\Z$ the rank function  defined by~$\rr(E)=r$ and~$\rr(A)=0$ for~$A\subsetneq E$. The corank-nullity polynomial of the ranked set~$S=(E,\rr)$ is
\[T_S(x,y)=(y-1)^{|E|-r}+(y^{|E|}-(y-1)^{|E|})(x-1)^r,\]
which is an~$(|E|,r)$-Brylawski polynomial with~$\deg_x T_S=r=\rr(E)$ and~$\deg_y T_S=|E|-1\geq |E|-\rr(E)$. 
As~$\deg_x T_S+\deg_y T_S=n+r-1$, assumption~(iii) of Theorem~\ref{th:MdMN} is not satisfied if~$r>2$.

The dual rank function~$\rr^*$ is defined by ~$\rr^*(\varnothing)=0$,
and~$\rr^*(A)=|A|-r$ for~$\varnothing\subsetneq A\subseteq E$.

The Newton polygon of the polynomial~$Y^{|E|-r}+X^r[(Y+1)^{|E|}-Y^{|E|}]$ is the convex hull of
vertices~$(0,|E|-r)$ and~$(r, i)$ for~$i\in\{0,1,\dotsc, |E|-1\}$ and this clearly cannot be expressed as a
Minkowski sum of smaller polygons (the direction vectors of the segments forming the convex hull --- a
triangle --- are~$(r,-|E|+r)$,~$(0,1)$ with multiplicity~$|E|-1$, and~$(-r,1-r)$). Hence~$T_S(x,y)$ is
irreducible.
\end{ex}

\begin{ex}\label{ex:2_nonneg_dual_neg} 
Let~$E$ be a non-empty finite set and for~$A\subset E$ define the rank function~$\rr$ by
\[
\rr(A)=\begin{cases} 0 & A=\varnothing\\
1 & \varnothing\subsetneq A\subsetneq E\\
|E| & A=E.\end{cases}
\]

The Tutte polynomial of~$S=(E,\rr)$ is
\begin{align*}T_S(x,y)&=(x-1)^{|E|}+\sum_{\varnothing\subsetneq A\subsetneq E}(x-1)^{|E|-1}(y-1)^{|A|-1}+1\\
& =(x-1)^{|E|}+(x-1)^{|E|-1}\left(\frac{y^{|E|}-1-(y-1)^{|E|}}{y-1}\right)+1.
\end{align*}
Here~$\deg_x T_S=|E|=\rr(E)$ and~$\deg_y T_S=|E|-2\geq |E|-\rr(E)=0$. As~$\deg_x T_S+\deg_y T_S=2|E|-2$ assumption~(iii) of Theorem~\ref{th:MdMN} fails when~$|E|>3$.
The dual rank function is defined by 
\[
\rr^*(A)=\begin{cases} 0 & A=\varnothing\\
1-|E\setminus A| & \varnothing\subsetneq A\subsetneq E\\
0 & A=E.\end{cases}
\]
The Newton polygon of the polynomial~$X^{|E|}+X^{|E|-1}\frac{(Y+1)^{|E|}-1-Y^{|E|}}{Y}$+1 is the convex hull
of vertices~$(|E|,0)$,~$(0,0)$ and~$(|E|-1, i)$ for~$i\in\{0,1,\dotsc, |E|-2\}$ (again this is a triangle and the
direction vectors of the sides are~$(-1,|E|-2)$,~$(-|E|+1,-|E|+2)$ and~$(|E|,0)$). This clearly cannot be
expressed as a Minkowski sum of smaller polygons. Hence~$T_S(x,y)$ is irreducible.
\end{ex}

\section{Probabilistic approach: generic Galois maximality}\label{sec:prob}
In this section, which is to a large extent of independent interest, we approach Conjecture~\ref{conj:BCM} from a probabilistic point of view: within suitable families of polynomials equipped with a ``height function" and originating from well identified rank functions (such as rank functions associated to connected matroids) we aim at obtaining an upper bound on the size of the set of ``pathological" elements (those having a non maximal Galois group) in the family as the height grows. This is achieved through Kowalski's sieve framework~\cite{Kow08} with a crucial appeal to Cohen's work~\cite{Coh72}. 
In the first subsection below we state and prove a generalized form of a uniform sieve bound due to Gallagher~\cite{Gal72} towards the celebrated (and recently solved~\cite{Bha25}) conjecture of van der Waerden~\cite{vdW36} from~$1936$.

\subsection{Generalizing a uniform version of a Theorem of Gallagher}\label{sec:sieve}

Gallagher~\cite{Gal72} considers, for fixed~$r\geq 1$ and for a growing parameter~$N\in\N_{\geq 1}$, the set
\[
E_r(N)\coloneqq \Big\{f(x)=x^r+\sum_{i=0}^{r-1}a_ix^i\colon a_i\in \Z,\, |a_i|\leq N,\, |\gal_\Q(f)|<r!\Big\}
\]
of polynomials~$f$ for which the Galois group of a splitting field over~$\Q$ (denoted~$\gal_\Q(f)$ in the above definition of~$E_r(N)$) is not maximal. A uniform version of the large sieve result of Gallagher~\cite{Gal72} states that~$|E_r(N)|/(2N+1)^r\ll r^3 \log N\cdot N^{-1/2}$ (for all~$N\geq 2$, all~$r\geq 1$, and with an absolute implied constant, see~\cite{Kow08}*{Th.~4.2}). The starting point of Gallagher's method is the identification of the set of monic~$\Z$-polynomials of degree~$r$ with~$\Z^r$ through fixing the canonical basis~$(1,x,\ldots,x^{r-1})$ of~$\Q$-polynomials of degree~$<r$.

We extend Gallagher's approach to more general linearly independent families of polynomials. Our result (Theorem~\ref{Th:Gallagher}) gives the same type of uniform upper bound on the proportion of pathological elements (still meant in the sense that the Galois group is not maximal) as in~\cite{Kow08}*{Th.~4.2}.

\medskip

For~$s\in\N$, consider a family~$(F_0,\ldots,F_s)$ of monic polynomials in~$\Z[x]$ with~$\deg F_s<\max\{\deg F_i\colon i\leq s-1\}\eqqcolon r$. For~$p$ an element in a set of prime numbers~$\mathcal P(r)$ of positive density (the definition of~$\mathcal P(r)$ may depend on~$r$ only and we assume that its density admits a positive lower bound uniform in~$r$), let~$(F_{0,p},\ldots,F_{s,p})$ be the reduction of~$(F_0,\ldots,F_s)$ modulo~$p$ (meaning that we reduce the coefficients of the polynomials~$F_i$ modulo~$p$). For any prime~$p\in\mathcal P(r)$, assume the following statements hold:
\begin{enumerate}
  \item[(H1)] the polynomials~$F_{i,p}$ are relatively prime for~$i\in\{0,\dotsc, s\}$, and~$(F_{0,p},\ldots,F_{s,p})$ is linearly independent over~$\F_p$ (in particular~$s\leq r$);
\item[(H2)] the family~$(F_{i,p}-\beta_iF_{s,p})_{0\leq i\leq s-1}$ has \emph{normal} gcd for all~$(\beta_i)\in\F_p^s$ in the sense of Cohen~\cite{Coh72}*{p.~95} (\emph{i.e.} the gcd of this family has at most one multiple irreducible factor which, if it exists, has multiplicity~$2$ and degree~$1$) and~$(F_{i,p})_{0\leq i\leq s}$ is \emph{not totally composite} in the sense of~\cite{Coh72}*{(2.2)} (a family~$(f_0,\ldots,f_s)\in \F_p[x]^{s+1}$ is \emph{totally composite} if there exists~$\psi=N/D\in \F_p(x)$, written in lowest degree terms, and~$(g_j)_{0\leq j\leq s}\in\F_p[x]^{s+1}$, such that~$\deg N>\deg D+1$,~$\ell\coloneqq \max_j\deg g_j> 1$ and~$f_i=D^\ell\cdot(g_i\circ \psi)$ for all~$i\in\{0,\ldots,s\}$).
\end{enumerate}

The main result of this subsection is the following extension of the uniform version of Gallagher's
Theorem~(\cite{Kow08}*{Th.~4.2}).

\begin{theorem}\label{Th:Gallagher}
With notation and assumptions as above, one has:
\[
\frac{\Big|\big\{(n_i)_{1\leq i\leq s}\in [-N,N]^{s}\colon |\gal_\Q\big( F_{0}+\sum_{i=1}^s n_i F_{i}\big)|<r! \big\}\Big|}{(2N+1)^s}\ll r^2\Big(1+\frac{1}{\log r}\Big)^{2s} \frac{\log N}{\sqrt N}, 
\]
for all integers~$s\geq 1$,~$r\geq 2$ and~$N\gg_r 1$, and with an absolute implied constant.
\end{theorem}

As already mentioned above, the proof proceeds by a sieving argument using primes in the set~$\mathcal P(r)$.
Before proving Theorem~\ref{Th:Gallagher}, we introduce the necessary objects and outline the sieve setup.

For a prime number~$p$, we let~$\pi_p\colon \Z^{s}\to\F_p^{s}$
be the reduction modulo~$p$ morphism that acts coordinate-wise. Let~$\lambda$ be a \emph{factorisation pattern} for polynomials of degree~$r$: given a partition~$\lambda$ of~$r$ (\emph{i.e.}~$\lambda=(a_1,\ldots,a_r)\in\Z_{\geq 0}^r$ with~$\sum_{i=0}^r ia_i=r$), a polynomial~$f$ of degree~$r$ has factorisation pattern~$\lambda$ if it has exactly~$a_i$ distinct irreducible factors of degree~$i$, for each~$i\in\{0,\ldots,r\}$ (in particular such an ~$\F_p$-polynomial is squarefree). For a prime number~$p$ and for~$\boldsymbol{\lambda}=\{\lambda_i\}_i$ a set of partitions of~$r$, define the following subset of~$\F_p^{s}$:
\[
\Omega_{\boldsymbol{\lambda},p}=\Big\{(\beta_i)_{1\leq i\leq s}\in \F_p^s\colon F_{0,p}+\sum_{i=1}^s\beta_i F_{i,p} \text{ has factorisation pattern~$\lambda\in {\boldsymbol \lambda}$ in~$\F_p[x]$}\Big\}.
\]
(We will simply write~$\Omega_{\lambda,p}$ if~$\boldsymbol\lambda=\{\lambda\}$ contains a single partition.)
In the spirit of~\cite{Kow08}*{\S 4.2}, we fix (for now) an auxiliary parameter~$L$ (to be eventually optimized) and consider the \emph{sieving problem} of finding an upper bound for the cardinality of
\[S(N,\boldsymbol\lambda,L)\coloneqq 
\{(n_i)_{1\leq i\leq s}\in [-N,N]^{s}\colon \forall p\in \mathcal P(r)\cap [1,L],\, \pi_p((n_i))\notin \Omega_{\boldsymbol\lambda,p}\}.
\]
As explained in~\cite{Kow08}*{\S 4.2}, if the reduction modulo a prime number~$p$ of~$F=F_{0}+\sum_{i=1}^s n_i
F_{i}$ (corresponding to~$(\pi_p(n_i))_{1\leq i\leq s}$) has factorisation pattern~$\lambda$ over~$\F_p$, then
at least one permutation of cycle type~$\lambda$ (\emph{i.e.} a product of~$a_1$ fixed points,~$a_2$ disjoint
transpositions, etc...)  belongs to~$\gal_{\Q}(F)$. Consequently if a polynomial~$F=F_{0}+\sum_{i=1}^s n_i
F_{i}$ has a Galois group over~$\Q$ not intersecting the conjugacy invariant subset~$c_{\boldsymbol\lambda}$
of~$\mathfrak S_r$ consisting of permutations of cycle type~$\lambda\in\boldsymbol \lambda$, then
the~$s$-tuple~$(n_i)_{1\leq i\leq s}$ belongs to~$S(N,\boldsymbol\lambda,L)$. In other words,
\[
  \Big|\big\{(n_i)_{1\leq i\leq s}\in [-N,N]^{s}\colon \gal_\Q( F_{0}+\sum_{i=1}^s n_i F_{i} ) \text{ does not
  intersect~$c_{\boldsymbol \lambda}$ in~$\mathfrak S_r$ }\big\}\Big|\leq\big |S(N,\boldsymbol\lambda,L)\big|.
\]
Moreover, if~$(c_{{\boldsymbol\lambda}_j})$ is a set of conjugacy invariant subsets such that no proper subgroup of~$\mathfrak S_r$ intersects every~$c_{\boldsymbol\lambda_j}$ then
\begin{equation}\label{eq:abstractupperbound}
\Big|\Big\{(n_i)_{1\leq i\leq s}\in [-N,N]^{s}\colon |\gal_\Q( F_{0}+\sum_{i=1}^s n_i F_{i} )|<r!\Big\}\Big|\leq\sum_j\big |S(N,\boldsymbol\lambda_j,L)\big|.
\end{equation}
The following \emph{sieve inequality} is a key ingredient to the proof of Theorem~\ref{Th:Gallagher}.

\begin{proposition}\label{prop:LS}
Let~$s\geq 1$ be an integer and let~$(F_0,\ldots,F_s)$ be an~$(s+1)$-tuple of monic polynomials in~$\Z[x]$ with~$\deg F_s<\max\{\deg F_i\colon i\leq s-1\}\eqqcolon r\geq 2$. Let~$\boldsymbol \lambda$ be a set of partitions of~$r$ and let~$\mathcal P(r)$ be a set of primes depending only on~$r$. For any positive integers~$N,L$, there exists constants~$\Delta(N,L)$ and~$H(L,\boldsymbol\lambda)$ depending only on~$(N,L)$ and~$(L,\boldsymbol\lambda)$, respectively, satisfying
\[
\big |S(N,\boldsymbol\lambda,L)\big|\leq \Delta(N,L)\cdot H(L,\boldsymbol\lambda)^{-1}.
\]
Moreover one has the bounds
\[
  \Delta(N,L)\leq (\sqrt{2N+1}+L)^{2s}\quad\text{and}\quad H(L,\boldsymbol\lambda)\geq \sum_{p\in\mathcal P(r)\cap [1,L]}\frac{|\Omega_{\boldsymbol\lambda,p}|}{|\F_p^s|}.
\]
\end{proposition}

\begin{proof}
Formally the first inequality as well as the lower bound on~$H$ are applications of Kowalski's sieve statement~\cite{Kow08}*{Prop.~3.5} (where the constants~$\Delta$ and~$H$ are also properly defined in a general context). The upper bound on~$\Delta$ is due to Huxley~\cite{Hux68}; see also~\cite{Kow08}*{Th.~4.1}.
\end{proof}

In view of Proposition~\ref{prop:LS}, our next task is to estimate the lower bound on~$H(L,\lambda)$. To this purpose we appeal to a result of Cohen~\cite{Coh72}*{Th.~3}.

\begin{lemma}[Cohen] \label{lem:Cohen} Keeping the notation as in~\ref{prop:LS}, assume that (H1) and (H2) hold for~$(F_0,\ldots,F_s)$. Then for every~$p\in\mathcal P(r)\cap (r,\infty)$,
\[
\frac{|\Omega_{\lambda,p}|}{|\F_p^s|}=T(\lambda)+O_r(p^{-1/2}),
\]
where~$T(\lambda)$ is the proportion of elements of the symmetric group~$\mathfrak S_r$ of cycle type~$\lambda=(a_1,\ldots,a_r)$ the partition of~$r$ with~$i$ contributing~$a_i$ times for all~$i\in\{1,\ldots,r\}$(precisely\footnote{For example the partition~$\lambda_{\mathrm{irr}}$ corresponding to irreducible polynomials is~$(a_1,\ldots,a_r)=(0,\ldots,0,1)$ and one has~$T(\lambda_{\mathrm{irr}})=1/r$.}~$T(\lambda)=(\prod_i{i^{a_i}a_i!})^{-1}$).
\end{lemma}

Note that the assumption~$p>r$ enables us to simplify some of the requirements in Cohen's work. Indeed, for such a~$p$, linear independence over~$\F_p$ is equivalent to linear independence over~$\F_p(x^p)$, as explained by Cohen~\cite{Coh72}*{p.~95}, and also~$F_s$ is automatically \emph{tame}~(which means that no zero of~$F_s$ has multiplicity divisible by~$p$~\cite{Coh72}*{statement of Th.~3}).

We are now ready to prove Theorem~\ref{Th:Gallagher} (the argument follows closely the proof of~\cite{Kow08}*{Th.~4.2}).

\begin{proof}[Proof of Theorem~\ref{Th:Gallagher}]
Imposing~$L\geq r^2$, say, we start by applying Lemma~\ref{lem:Cohen}. First we have for any partition~$\lambda$ of~$r$,
\[
\sum_{p\in\mathcal P(r)\cap [1,L]}\frac{|\Omega_{\lambda,p}|}{|\F_p^s|}\geq 
\sum_{p\in\mathcal P(r)\cap (r,L]}\frac{|\Omega_{\lambda,p}|}{|\F_p^s|}=T(\lambda)\Big(|\mathcal P(r)\cap (r,L]|\Big)+O_r\Big(\sum_{p\leq L}\frac 1{\sqrt{p}}\Big),
\]
where, in the error term, the sum extends to primes up to~$L$.
Moreover, summation by parts yields the following upper bound:
\[
\sum_{p\leq L}\frac 1{\sqrt{p}}\leq \frac{\sqrt L}{\log L}.
\]

Let~$\boldsymbol\lambda$ be a finite set of partitions of~$r$ and let~$T(\boldsymbol \lambda)=\sum_{\lambda\in\boldsymbol\lambda}T(\lambda)$.
By Proposition~\ref{prop:LS} we deduce, invoking the Prime Number Theorem (which implies~$|\mathcal P(r)\cap (r,L]|\gg L/\log L$ since~$L\geq r^2$), that we have for some constant~$C_r>0$ depending only on~$r$ and as soon as~$\sqrt L> \max(r,C_r/T(\boldsymbol\lambda))$,
\[
H(L,\boldsymbol\lambda)^{-1}\gg T(\boldsymbol\lambda)|\mathcal P(r)\cap (r,L]|-C_r \frac{\sqrt L}{\log L}\gg T(\boldsymbol\lambda)\frac L{\log L}\Big(1-\frac{C_r}{T(\boldsymbol\lambda)\sqrt L}\Big),
\]
with absolute implied constants. In turn we obtain
\[
|S(N,\boldsymbol\lambda,L)|\ll (\sqrt{2N+1}+L)^{2s}\cdot (T(\boldsymbol\lambda))^{-1}\frac{\log L}{L}\Big(1-\frac{C_r}{T(\boldsymbol\lambda)\sqrt L}\Big)^{-1}.
\]
Now for ~$N\gg \max(T(\boldsymbol\lambda)^{-2}r^4,T(\boldsymbol\lambda)^{-3}C_r^4)$, we choose~$L=T(\boldsymbol\lambda)\sqrt{2N+1}$. We obtain
\begin{equation}\label{eq:sieveupperbound}
|S(N,\boldsymbol\lambda,L)|\ll T(\boldsymbol\lambda)^{-2}(1+T(\boldsymbol\lambda))^{2s}N^{s-\frac 12}\log N.
\end{equation}

To conclude, we follow the end of the proof of~\cite{Kow08}*{Th.~4.2} by invoking~\cite{Gal72}*{Lemma p.~98}: no proper transitive subgroup of~$\mathfrak S_r$ contains both a transposition and a~$q$-cycle of prime length~$q>r/2$. Therefore it is enough to consider the following families of factorization patterns:
\begin{itemize}
    \item~$\lambda_{\mathrm{irr}}$ the trivial partition~$(a_1,\ldots,a_r)$ of~$r$ with~$a_r=1$, which detects the transitivity of the Galois action,
    \item~$\boldsymbol\lambda_{\mathrm{tr}}$ the set of partitions~$(a_1,\ldots,a_r)$ of~$r$ satisfying~$a_2=1$ and~$a_{2i}=0$ for~$i\neq 1$, which detects an element of the Galois group acting as a transposition,
    \item~$\boldsymbol\lambda_{\mathrm{prime}}$ the set of partitions~$(a_1,\ldots,a_r)$ of~$r$ with~$a_q=1$ for some prime number~$q> r/2$, which detects an element of the Galois group acting as a~$q$-cycle.
\end{itemize}
One has~$T(\lambda_{\mathrm{irr}})= r^{-1}$ and the asymptotics~(\cite{Gal72}*{p.~99}):
\[
T(\boldsymbol\lambda_{\mathrm{tr}})\sim \frac{\log 2}{\log r},\qquad
T(\boldsymbol\lambda_{\mathrm{prime}})\sim \frac{1}{\sqrt{2\pi r}}\qquad (r\to\infty) .
\]
In particular, $\max(T(\lambda_{\mathrm{irr}})^{-2},
T(\boldsymbol\lambda_{\mathrm{tr}})^{-2},
T(\boldsymbol\lambda_{\mathrm{prime}})^{-2})\ll r^2$. Combining this with~\eqref{eq:abstractupperbound} and~\eqref{eq:sieveupperbound}, the proof is complete.
\end{proof}

\subsection{Application to linear combinations of Tutte polynomials of connected matroids}
We restrict to the case~$s=2$: let~$T_{M_1}(x,y)$ and~$T_{M_2}(x,y)$ be distinct Tutte polynomials of connected matroids~$M_1, M_2$ of rank~$r\geq 2$ and size~$n\geq r+1$. By Corollary~\ref{cor:dMMN} (originally~\cite{MdMN} in this case)~$T_{M_1}$ and~$T_{M_2}$ are irreducible in~$\Z[x,y]$. The fact that~$T_{M_1}$ and~$T_{M_2}$ are distinct is equivalent to their~$\Q$-linear independence. From~\cite{GJS}*{\S 5} the conditions~$r\geq 2$ and~$n\geq r+1$ guarantee that the dimension of the~$\Q$-span of Tutte polynomials of connected matroids of size~$n$ and rank~$r$ is at least~$2$ and therefore one can indeed pick two linearly independent Tutte polynomials of connected matroids of size~$n$ and rank~$r$. Consequently~$T_{M_1}$ and~$T_{M_2}$ are distinct irreducible elements of~$\Z[x,y]$.  We will state and prove the main result of this section (Theorem~\ref{Th:Gallagher}) in the more general setting where we only require that~$T_1$ and~$T_2$ be copime squarefree in~$\Z[x,y]$ and monic of the same degree as~$x$-polynomials with coefficients in~$\Z[y]$.
 
We start by establishing the following result explaining that suitable specializations of~$T_1$ and~$T_2$ at integral values of~$y$ yield triples~$(F_0,F_1,F_2)$ satisfying assumptions (H1) and (H2) of~\S\ref{sec:sieve}.
 
\begin{proposition}\label{prop:affineplane} Let~$r\in\N_{\geq 2}$ and let~$T_1(x,y)$ and~$T_2(x,y)$ be elements of~$\Z[x,y]$ that are coprime, squarefree, and monic of the same degree~$r$ seen as~$x$-polynomials with coefficients in~$\Z[y]$. For big enough~$t\in \Z$ (depending only on~$(T_1,T_2)$) there exists a constant~$C(T_1,T_2,t)$ (depending only on~$(T_1,T_2,t)$) so that for every prime~$p\geq C(T_1,T_2,t)$ the polynomials~$T_{1,p}(x,t_p)$ and~$T_{2,p}(x,t_p)$ are squarefree and coprime (where~$t_p=t\bmod p$ and, for any~$g\in \Z[x,y]$, we define~$g_p$ to be the element of~$\F_p[x,y]$ obtained by reducing the coefficients of~$g$ modulo~$p$). Moreover, setting
 \[
 F_1(x,y)=T_1(x,y)\,,\qquad F_2(x,y)=T_1(x,y)-T_2(x,y)\,,\qquad F_0(x,y)=xF_2(x,y),
 \]
then for every~$p\geq C(T_1,T_2,t)$, the triple~$(F_{i,p}(x,t_p))_{0\leq i\leq 2}$ satisfies assumptions (H1) and (H2) of~\S\ref{sec:sieve}.
 \end{proposition}
 
 \begin{proof}
 
We consider~$P(x,y)=T_1(x,y)T_2(x,y)$, the product of~$T_1$ and~$T_2$. Since~$T_1$ and~$T_2$ are distinct and irreducible in the UFD~$\Z[x,y]$, the polynomial~$P(x,y)$ is squarefree in~$\Z[x,y]$ (and also in~$(\Q(y))[x]$) and so its~$x$-discriminant (an element of~$\Z[y]$) is non zero. Therefore for big enough~$t\in \Z$ (depending only on~$(T_1,T_2)$), the specialization~$P(x,t)$ is squarefree; in other words~$T_1(x,t)$ and~$T_2(x,t)$ are coprime and squarefree. For every such~$t$ the discriminant of the~$x$-polynomial~$T_1(x,t)T_2(x,t)$ is non zero and therefore for every big enough prime~$p$ (depending only on~$(T_1,T_2,t)$), this discriminant is non zero modulo~$p$, meaning that~$T_{1,p}(x,t_p)$ and~$T_{2,p}(x,t_p)$ are squarefree and coprime.
 
Below we will write~$F_{i,p}$ as shorthand for~$F_{i,p}(x,t_p)$ ($0\leq i\leq 2$) for brevity. We can now check that assumptions (H1) and (H2) are satisfied by the relevant triples~$(F_{i,p})_{0\leq i\leq 2}$.
 
 \begin{itemize}
     \item[(H1)]  Since~$T_{1,p}(x,t_p)$ and~$T_{2,p}(x,t_p)$ are coprime, then~$\mathrm{gcd}(F_{1,p},F_{2,p})=1$ and therefore~$(F_{i,p})_{0\leq i\leq 2}$ is a family of relatively prime polynomials.
     
     Let us check that the linear independence assumption is satisfied. Since~$F_{1,p}$ and~$F_{2,p}$ are coprime, they are not colinear. Assume by contradiction that for some~$\alpha,\beta\in\F_p$ one has
 \[
 F_{0,p}=\alpha F_{1,p}+\beta F_{2,p}.
 \] 
 Recombining, we obtain
 \[
 (x-(\alpha+\beta))T_{1,p}(x,t_p)=(x-\beta)T_{2,p}(x,t_p).
 \] 
 This contradicts the fact that~$T_{1,p}(x,t_p)$ and~$T_{2,p}(x,t_p)$ are coprime, squarefree, and have degree~$r\geq 2$.
     \item[(H2)]
     We first check the \emph{normality} assumption. For any~$(\beta_0,\beta_1)\in\F_p^2$, we have
 \begin{align*}
\mathrm{gcd}(F_{0,p}-\beta_0 F_{2,p},F_{1,p}-\beta_1 F_{2,p})&=\mathrm{gcd}((x-\beta_0)F_{2,p},F_{1,p}-\beta_1 F_{2,p})\\
&=
\mathrm{gcd}(x-\beta_0,F_{1,p}-\beta_1 F_{2,p}),
\end{align*}
where, for the last step, we use the fact that~$F_{2,p}$ and~$F_{1,p}$ are coprime. This implies that the gcd considered is normal. 

\medskip

Finally, assume by contradiction that~$(F_{i,p})_{0\leq i\leq 2}$ is totally composite, then in particular
there exists polynomials~$g_0$ and~$g_2$ in~$\F_p[x]$ and a rational function~$\psi\in\F_p[x]$ of positive
degree such that
 \[
 F_{0,p}=xF_{2,p}=D^\ell g_0(\psi)\,,\qquad F_{2,p}=D^\ell g_2(\psi).
 \]
This leads to
 \[
 x=(h\circ \psi)(x)
 \]
where~$h=g_0/g_2$. Taking degrees, we obtain~$1=(\deg h)(\deg \psi)$. However both factors on the right hand side are integers and~$\deg \psi\geq 2$ by assumption; a contradiction.
\end{itemize}

 \end{proof}
 
 We are now ready to draw the following consequence of Theorem~\ref{Th:Gallagher} and Proposition~\ref{prop:affineplane}. This establishes in particular that, generically degree one~$\Z[x]$-linear combinations of two linearly independent Tutte polynomials of connected~$(n,r)$-matroids have maximal Galois group over~$\Q(y)$.
 
 \begin{theorem}
Let~$T_1(x,y)$ and~$T_2(x,y)$ be elements of~$\Z[x,y]$ that are coprime, squarefree, and monic of the same degree~$r\geq 2$ seen as~$x$-polynomials with coefficients in~$\Z[y]$ (for example one can take~$T_1(x,y)$ and~$T_2(x,y)$ to be distinct Tutte polynomials of connected~$(n,r)$-matroids with~$r\geq 2$ and~$n\geq r+1$). Then one has
 \[
\frac{\Big|\big\{(n_1,n_2)\in \{-N,\ldots,N\}^{2}\colon |G_{\Q,y}\big((x+n_1)T_{1}-(x+n_2)T_2\big)|<r! \big\}\Big|}{(2N+1)^2}\ll r^2 \frac{\log N}{\sqrt N}, 
\]
for all~$N\gg _r 1$ and with an absolute implied constant.
 \end{theorem}
 
 \begin{proof}
 As in Proposition~\ref{prop:affineplane} set~$F_0(x,y)=x(T_1(x,y)-T_2(x,y))$,~$F_1(x,y)=T_1(x,y)$, and~$F_2(x,y)=T_1(x,y)-T_2(x,y)$. Picking~$t$ big enough (again in the sense of Proposition~\ref{prop:affineplane}) the assumptions of Theorem~\ref{Th:Gallagher} (with~$s=2$) are satisfied for the choice~$\mathcal P(r)=\{p\text{ prime}\colon p\geq\max (C(T_1,T_2,t),r+1)\}$ which depends on~$(T_1,T_2,t)$ only and has density~$1$. We conclude by applying Theorem~\ref{Th:Gallagher} and a specialization argument (Proposition~\ref{prop:spec}, where one additionally notes that if~$f(x,y)\in \Z[x,y]$ has~$x$-degree~$d$ and if~$t,p\in \Z$ are such that~$p$ is prime and~$f(x,t)\bmod p$ is irreducible of degree~$d$ in~$\F_p[x]$, then~$f(x,y)$ is irreducible seen as an~$x$-polynomial with coefficients in~$\Q(y)$).
 \end{proof}
 
\section{Monodromy computations for some ranked sets and specific families of \texorpdfstring{$2$}{2}-connected graphs}\label{sec:fam}

In this section we first consider possible extensions of Conjecture~\ref{conj:BCM} to general ranked sets, and
next we prove that the conjecture holds for some families of connected matroids.

\subsection{Examples of monodromy group of Tutte polynomials of ranked sets}\label{sec:4.1}

In order to prove that Conjecture~\ref{conj:BCM} holds in a number of cases, we will use the same strategy as
for the proof of Theorem~\ref{Th:Gallagher} by showing that the Galois group considered, seen as a permutation
group, contains elements having a certain cycle type. Typically the final step of the proof relies on a
classical specialization argument (also used in the proof of Theorem~\ref{Th:Gallagher}). We state the
version, restricted to a UFD, of~\cite{BCM}*{Prop.~4}, itself a simplified form of~\cite{Lang}*{VII, Th.~2.9}).

\begin{proposition}\label{prop:spec}
Let~$R$ be a UFD equipped with a ring homomorphism~$\psi\colon R\to L$ to a field~$L$. Let~$f(X)\in R[X]$ be monic irreducible and assume that~$\psi(f)$ (the element of~$L[X]$ obtained by mapping~$\psi$ to the coefficients of~$f$) is separable. Then~$\gal_L(\psi(f))$ is a subgroup of~$\gal_k(f)$, where~$k$ denotes the field of fractions of~$R$. 
\end{proposition}

\begin{remark}\label{rem:basefieldext}
Let~$K/k$ be an algebraic field extension. Let~$T(x,y)\in k[x,y]$ be squarefree and non-constant as an~$x$-polynomial with coefficients in~$k[y]$. Fix an algebraic closure~$\overline{K(y)}$ of~$K(y)$; this is also an algebraic closure of~$k(y)$. Let~$(\alpha_i)_{1\leq i\leq r}$ be the roots of~$T$ in~$\overline{K(y)}$. Then~$k(y)((\alpha_i)_i)$ (resp.~$K(y)((\alpha_i)_i)$) is a splitting field of~$T$ over~$k(y)$ (resp.~$K(y)$) and one has an injective group morphism
\begin{equation*}
%\gal_{K(y)}(T)\to\gal_{k(y)}(T)\,,\qquad \sigma\mapsto \sigma_{|k(y)((\alpha_i)_i)}.
%\begin{aligned}
\setlength{\arraycolsep}{0pt}
\renewcommand{\arraystretch}{1.2}
\begin{array}{ c c c }
  \gal_{K(y)}(T) & {} \longrightarrow {} & \gal_{k(y)}(T)\\
  \sigma & {} \longmapsto {} & \sigma_{|k(y)((\alpha_i)_i)}.
\end{array}
%\end{aligned}
\end{equation*}
As a consequence, maximality of~$\gal_{K(y)}(T)$ (meaning it has order~$r!$) implies  maximality of~$\gal_{k(y)}(T)$. Moreover, if~$K/k$ has finite degree, then~$\gal_{k(y)}(T)\simeq \gal_{K(y)}(T)$. Therefore, in the sequel, even though we fix, at times, the field of constants that we work over (\emph{e.g.}~$\Q$ or~$\C$), one should keep in mind that under the above assumptions the statements still hold after a finite constant field extension or by considering a subfield over which the base field is algebraic. 
\end{remark}

In the remainder of \S\ref{sec:4.1}, we give examples showing that Conjecture~\ref{conj:BCM} extends to some but not all general rank functions.

\begin{remark}
To give a first low degree example in the case~$R=\Z$, we go back to the first example in Remark~\ref{rem:exGordon}. Specializing the Tutte polynomial~$T_S(x,y)=x^3+2x^2+y^2+3xy$ (see~\cite{gordon15}*{p.~23}) at~$y=1$ and reducing modulo~$5$, we obtain the product of irreducibles~$(x+1)(x^2+3)$ in~$\F_5[x]$. We deduce that the Galois group of~$T_S(x,y)$ over~$\Q(y)$ is isomorphic to~$\mathfrak S_3$ (up to isomorphism, it is the unique transitive permutation group of degree~$3$ that contains a transposition).
\end{remark}

The first example below corresponds to a rank function assuming two values only. In this case the Galois group is not maximal (generically it has index~$\geq 3$ inside the relevant symmetric group). The second example corresponds to a rank function assuming three values. In this case we exhibit subfamilies for which we show that the Galois group is~maximal. 

\subsubsection{A rank function assuming two values}\label{sec:ex1} Example~\ref{ex:1_nonneg_dual_neg} gives an example of a rank function whose dual takes negative values, and whose corank-nullity polynomial is the polynomial
\[T_{n,r}(x,y)=(y-1)^{n-r}+(y^{n}-(y-1)^{n})(x-1)^r,\]
where~$n\geq r\geq 1$.
Let~$R$ be a UFD and let~$\varphi\colon \Z\to R$ be the natural ring homomorphism the kernel of which is generated by the characteristic of~$R$. Applying~$\varphi$ to the coefficients, we consider the  polynomial~$T$ as an element of~$R[x,y]$. We let~$k$ be the fraction field of~$R$.

\begin{proposition}\label{prop:non-max}
 As a polynomial in~$x$,~$T_{n,r}(x,y)$ is irreducible over~$k(y)$. Assuming moreover that the characteristic of~$R$ does not divide~$r$, the Galois group~$G_{k,y}(T_{n,r})$ has order at most~$r\varphi(r)$. If~$r\geq 4$, then the index of~$G_{k,y}(T_{n,r})$ in~$\mathfrak S_r$ is at least~$3$. 
 In particular~$G_{k,y}(T_{n,r})$ is neither isomorphic to~$\mathfrak S_r$ nor to~$\mathfrak A_r$ as soon as~$r\geq 4$.  
\end{proposition}

\begin{proof}
By invariance of the Galois group structure (and in particular reducibility properties) when factoring out a non-zero constant term, taking the reciprocal or performing linear transformations, we can replace~$T_{n,r}(x,y)$ by 
\[
U_{n,r}(x,y)=x^r-\frac{(y-1)^{n}-y^{n}}{(y-1)^{n-r}}.
\]
The polynomial~$U_{n,r}$ has coefficients in~$k(y)$ and its constant term is not a constant multiple of any~$g(y)^\ell$ for some~$g(y)\in k(y)$ and~$\ell\in\N_{\geq 2}$. Indeed the numerator of such a~$g(y)^\ell$ has multiple roots, which is not the case for~$(y-1)^{n}-y^{n}$ as one checks by computing the~$y$-derivative. Applying~\cite{Lang}*{Chap.~6, Th.~9.1}, we deduce that~$U_{n,r}$ is irreducible over~$k(y)$.

Furthermore~\cite{Lang}*{Chap.~6, \S9, Ex.~2} constructs, under the stated assumptions, an injective morphism from~$G_{k,y}(U_{n,r})$ to
\[
\left\{\begin{pmatrix}
a & b\\
0 & 1
\end{pmatrix}\colon b\in \Z/r\Z,\, a\in (\Z/r\Z)^\times\right\},
\]
which concludes the proof.
\end{proof}

\subsubsection{A rank function assuming three values}\label{sec:ex2} 
The ranked set of Example~\ref{ex:2_nonneg_dual_neg} has corank-nullity polynomial equal to
\[T_{n}(x,y)=(x-1)^{n}+(x-1)^{n-1}\left(\frac{y^n-1-(y-1)^n}{y-1}\right)+1,\]
where~$n\geq 1$.

Setting~$X=x-1$ we obtain 
\[
T_n(X,y)=X^n+\left(\frac{y^n-1-(y-1)^n}{y-1}\right)X^{n-1}+1.
\]
For a suitable factorization of~$n$, the following statement asserts the maximality of the Galois group of~$T_n$ over~$\Q(y)$.
\begin{proposition}\label{prop:p1p2}
Let~$n=p_1p_2$ be the product of two odd primes such that the (multiplicative) order of~$2$ modulo~$p_i$ does not divide~$p_j-1$ (for any~$i\neq j$). Then the Galois group of the splitting field of~$T_n(X,2)$ over~$\Q$ is isomorphic to~$\mathfrak S_n$ and therefore~$G_{\Q,y}(T_n(x,y))\simeq \mathfrak S_n$.
\end{proposition}

\begin{proof}
First note that
\[
T_n(X,2)=X^n+(2^n-2)X^{n-1}+1.
\]
By assumption, one has~$n\geq 15$ and therefore~\cite{Har}*{Th.~1} implies that~$T_n(X,2)$ is irreducible over~$\Q$. In particular~$T_n(X,y)$ is irreducible over~$\Q(y)$.

Moreover~$\mathrm{gcd}(n,2^n-2)=1$. Indeed, for~$\{i,j\}=\{1,2\}$, one has
\[
2^n-2=(2^{p_i})^{p_j}-2\equiv 2(2^{p_j-1}-1)\not\equiv 0 \bmod p_i,
\]
since the order of~$2$ modulo~$p_i$ does not divide~$p_j-1$, by assumption. We conclude that the Galois group of the splitting field of~$T_n(X,2)$ over~$\Q$ is isomorphic to~$\mathfrak S_n$ by invoking~\cite{Osa}*{Th.~1}. We finish the proof by applying Proposition~\ref{prop:spec}.
\end{proof}

From Proposition~\ref{prop:p1p2} we can deduce that there are infinitely many values of~$n$ for which the Tutte polynomial of the ranked set in Example~\ref{ex:2_nonneg_dual_neg} has maximal Galois group over~$\Q(y)$. The following statement gives a strong form of this fact.

\begin{corollary}\label{cor:pos-proportion}
For large enough~$t\geq 2$, one has the following lower bound (that holds with an absolute implied constant):
\[
\#\{n\leq t\colon G_{\Q,y}(T_n(x,y))\simeq \mathfrak S_n\}\gg \frac{t^{0.16}}{\log t}.
\]
\end{corollary}

\begin{proof}
By Proposition~\ref{prop:p1p2}, a lower bound for the quantity investigated is
\begin{equation}\label{eq:lb1}
\#\{n=p_1p_2\leq t\colon \mathrm{ord}_{p_i}(2)\nmid (p_j-1)\text{ for }\{i,j\}=\{1,2\}\},
\end{equation}
where~$p_1,p_2$ are odd prime numbers and~$\mathrm{ord}_{p_i}(2)$ denotes the multiplicative order of~$2$
modulo~$p_i$. Modulo a prime at least~$5$, we know that~$2$ cannot have order~$2$. Also, by
Fermat,~$\mathrm{ord}_{p_i}(2)\mid (p_i-1)$ for any odd prime~$p_i$, and therefore~\eqref{eq:lb1} is
bounded from below by
\[
\#\{n=p_1p_2\leq t\colon p_i\neq 3\,,p_i\equiv 3\bmod (4) ,\,(i=1,2),\, \mathrm{gcd}(p_1-1,p_2-1)=2\}.
\]
We rewrite this quantity as follows:
\begin{equation}\label{eq:lb2}
\sum_{\substack{7\leq p_1\leq t\\ p_1\equiv 3\bmod 4}}\#\{7\leq p_2\leq t/p_1\colon p_2\not\equiv 1\bmod q\text{ if }q=4 \text{ or }q\mid (p_1-1),\, q\text{ odd prime}\}.
\end{equation}
To estimate the general term of this sum we need to count, for given~$p_1$, the prime numbers~$p_2$ that are less than~$t/p_1$ and satisfy a linear system of type~$\{p_2\equiv a_i\bmod q_i,\,\forall i\in\{0,\ldots,r\}\}$ where~$q_0=4$,~$q_1,\ldots, q_r$ are the distinct odd prime divisors of~$p_1-1$ and~$a_i$ is an invertible class distinct from~$1$ modulo~$q_i$ for~$i\geq 1$. By the Chinese Remainder Theorem each such system is equivalent to a single congruence~$p_2\equiv a\bmod (q_0\cdots q_r)$. One knows~\cite{Xyl}*{Th.~1.1} that there exists a prime number~$p_2\ll (q_1\cdots q_r)^{5.18}$ satisfying such a congruence (with an absolute effectively computable implied constant). This means that if one restricts the sum in~\eqref{eq:lb2} to primes~$p_1\leq t^{1/6.2}$, for big enough~$t$, then there is a suitable prime~$p_2\leq t/p_1$. In particular the sum~\eqref{eq:lb2} is bounded from below, for big enough~$t$, by
\[
\pi(t^{1/6.2}; 4,3),
\]
the count of prime numbers that are~$3$ modulo~$4$ and less than~$t^{1/6.2}$.
We conclude by invoking the Prime Number Theorem in arithmetic progressions and the fact that~$1/6.2\simeq 0.1613$.
\end{proof}

In the rest of Section~\ref{sec:fam}, we prove that Conjecture~\ref{conj:BCM} holds for several families
of~$2$-connected graphs. In some cases we do not prove the full force of the maximality result predicted by
the conjecture but we obtain information about the Galois action that goes beyond the transitivity granted by
our irreducibility statement (Corollary~\ref{cor:dMMN}).

\subsection{The case of an \texorpdfstring{$n$}{n}-cycle}

Let~$n\geq 3$ and let~$C_n$ be the cycle on~$n$ vertices, a~$2$-connected graph (the associated matroid is connected) with Tutte polynomial given by
\[
T_{C_n}(x,y)=\sum_{i=1}^{n-1}x^i+y.
\]
In this section, we prove the conjecture of Bohn--Cameron--M\"uller for the graphic matroid associated with the~$n$-cycle~$C_n$ where~$n\geq 3$.

\begin{theorem}\label{th:CnCase}
Let~$n\geq 3$ and let~$k$ be a field of characteristic~$0$ or~$p\nmid n(n-1)$. We have~$G_{k,y}(T_{C_n})
\simeq\mathfrak S_{n-1}$. In particular Conjecture~\ref{conj:BCM} holds for~$C_n$.
\end{theorem} 

Below we give a proof valid over any field satisfying the assumptions of the statement. For the case~$k=\Q$ (and therefore, over any number field by Remark~\ref{rem:basefieldext}) we will derive an independent proof as a consequence of Lemma~\ref{lem:Selmer}. Before proving Theorem~\ref{th:CnCase} we consider a slight modification of~$T_{C_n}$ better suited for computing its Galois group over~$k(y)$ with~$k$ as in the statement of Theorem~\ref{th:CnCase}.

\begin{lemma} \label{lem:Elkies}
For~$n$ and~$k$ as in Theorem~\ref{th:CnCase}, let~$p(x)=(x^n-1)/(x-1)\in k[x]$. Then, setting~$P(x,y)=p(x)-y$, we have~$G_{k,y}(P)\simeq \mathfrak S_{n-1}$.
\end{lemma}

\begin{proof}[Proof of Lemma~\ref{lem:Elkies}]
We present the argument given by N. Elkies~\cite{El}. It uses a result that goes back to Hilbert (\cite{Ser}*{Th.~4.4.1}) under the generalized form proved in Serre's book~\cite{Ser}*{Th.~4.4.5} and which asserts the following. Let~$f\in k[x]$ be of  
degree~$d$, and suppose the roots~$x_1,\dotsc, x_{d-1}$ of the derivative~$f'$ are pairwise distinct and that the images~$y_j\coloneqq f(x_j)$ 
are also pairwise distinct (\emph{i.e.}~$f$ is a so-called \emph{Morse function}), then the Galois group of~$f(x)-y$ over~$k(y)$ is isomorphic to 
$\mathfrak S_{d}$.

To see that this holds for~$p(x)$ as in the statement, note that the discriminant of~$p(x)-y$ with respect
to~$x$ has precisely the~$y_j$'s as roots (this is seen \textit{e.g.} by using the defining formula for the
discriminant which is, up to sign, the resultant of a polynomial and its derivative). In fact the computation
\[
\frac{\mathrm{d}}{\mathrm{d}x}\big((x-1)(p(x)-y)\big)=(p(x)-y)+(x-1)p'(x)
\]
shows that~$(x-1)(p(x)-y)$ and its~$x$-derivative have a common root if and only if~$y$ is one of the~$y_j$'s, or~$y=n$ (in the latter case the common root is~$x_0=1$). Thus if one wants to detect multiplicities in the~$y_j$'s by using the formula for the discriminant of a trinomial (see \emph{e.g.}~\cite{Swa}*{Th.~2})
\[
\disc_x((x-1)(p(x)-y))=\disc_x(x^n-yx-1+y)=\pm( (n-1)^{n-1}y^n-n^n(y-1)^{n-1}),
\]
one first has to divide this~$y$-polynomial by the highest power of~$(y-n)$ that it is a multiple of. One easily sees that this maximal power is~$(y-n)^2$ and we deduce that the roots of 
\[
q(y)\coloneqq \frac{(n-1)^{n-1}y^n-n^n(y-1)^{n-1}}{(y-n)^2}
\]
are the~$y_j$'s counted with multiplicity. To finish the proof one computes the discriminant of the polynomial~$q$ and checks that it does not vanish. Elkies computes~(\cite{El}) this discriminant and shows that, up to sign, it equals~$2\Delta_{n}\Delta_{n-1}$ (here, for any positive integer~$m$ we define~$\Delta_m=m^{(m-1)(m-3)}$) which is non zero by our assumption on the characteristic of~$k$.
\end{proof}

\begin{proof}[Proof of Theorem~\ref{th:CnCase}]
In the notation of Lemma~\ref{lem:Elkies}, one has~$P(x,1-y)=T_{C_n}(x,y)$ and therefore we deduce that~$G_{k,1-y}(P)=G_{k,y}(T_{C_n})\simeq \mathfrak S_{n-1}$ since~$k(y)=k(1-y)$.
\end{proof}

\subsection{The case of uniform matroids}
Let~$U_{a,b}$ be the uniform matroid, which is defined on groundset~$\{1,\ldots,b\}$ and has bases all subsets of size~$a$: the rank of~$A\subseteq \{1,\dots, b\}$ is equal to~$|A|$ if~$|A|\leq a$ and equal to~$a$ if~$a<|A|\leq b$. The Tutte polynomial of~$U_{a,b}$ is then given by 
\[
T_{U_{a,b}}(x,y)=\sum_{i=0}^{a}\binom{b}{i}(x-1)^{a-i}+\sum_{j=a+1}^b \binom{b}{j}(y-1)^{j-a}.
\]
Assume further that~$0<a<b$, for which~$U_{a,b}$ is connected (Brylawski~\cite{B72}*{Cor.~7.14} showed that the coefficient of~$x$ in~$T_{U_{a,b}}(x,y)$ is equal to~$\binom{b-2}{a-1}$).

\begin{theorem}\label{th:UnifMat}
Assume~$a\geq 2$. 
\begin{enumerate}
  \item\label{itunifmat1} One has~$G_{\Q,y}(T_{U_{a,b}})\simeq \mathfrak S_a$ if~$b$ is big enough in terms of~$a$.

  \item\label{itunifmat2}   If~$b=a+p$ for some prime number~$p>a$ then~$G_{\Q,y}(T_{U_{a,b}})$ acts as a primitive permutation group on the roots of~$T_{U_{a,b}}$ in a fixed algebraic closure of~$\Q(y)$. If in addition~$a$ is composite, then the action of~$G_{\Q,y}(T_{U_{a,b}})$ is doubly transitive.
\end{enumerate}
\end{theorem}

\begin{remark}
  The reason why the second statement of Theorem~\ref{th:UnifMat} is included in addition to the first one is because of uniformity issues. As we explain in the proof below, Theorem~\ref{th:UnifMat}\eqref{itunifmat1} is quite directly deduced from work of Filaseta and Moy~\cite{FiMo} (recently complemented by Klahn and Technau~\cite{KlTe}), however no explicit growth rate for~$b$ as a function of~$a$ is made explicit in~\cites{FiMo,KlTe} and it does not seem clear how to extract this information from their approach. Celebrated results (see \emph{e.g}~\cite{DiMo}*{Chap.~4 and~\S7.7}) in group theory show that primitivity (and all the more so for, double transitivity) is not a property shared by many subgroups of the symmetric group.
\end{remark}

\begin{proof}[Proof of Theorem~\ref{th:UnifMat}]

\eqref{itunifmat1} We note that 
\[
q^\star_{a,b}(x-1)\coloneqq T_{U_{a,b}}(x,1)=\sum_{i=0}^a
\binom{b}{i}(x-1)^{a-i}\,.
\]
This polynomial is the reciprocal evaluated at~$x-1$ of the polynomial denoted~$q_{a,b}(x)$ in~\cite{FiMo}*{p.~295}. For fixed~$a\geq 2$, one can apply~\cite{FiMo}*{Theorem 1} (complemented, for~$a=6$, by~\cite{KlTe}*{Th.~2}) which asserts that for big enough~$b$ (as a function of~$a$), the Galois group of~$q_{a,b}(x-1)$ over~$\Q$ is isomorphic to~$\mathfrak S_a$. The same holds for~$q^\star_{a,b}(x)$, and we conclude by Proposition~\ref{prop:spec}.

\medskip
\eqref{itunifmat2} Let~$\alpha_1,\ldots,\alpha_a$ be the roots of~$T_{U_{a,b}}$ in an algebraic closure of~$\Q(y)$. 
By definition~$G_{\Q,y}(T_{U_{a,b}})\simeq \gal(\Q(y)((\alpha_i)_i)/\Q(y))$. For simplicity this group will be denoted~$G$ and we will write
\[
f(x)\coloneqq \sum_{i=0}^{a}\binom{b}{i}(x-1)^{a-i}\,,\qquad g(y)\coloneqq \sum_{j=a+1}^b \binom{b}{j}(y-1)^{j-a},
\]
so that~$T_{U_{a,b}}(x,y)=f(x)+g(y)$. The following diagram summarizes the situation.
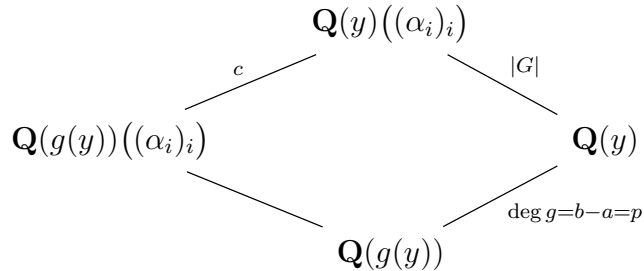
\[
\begin{tikzcd}
 & \Q(y)\big((\alpha_i)_i\big) \arrow[ld, "c"', dash] \arrow[rd, "|G|", dash] & \\
\Q(g(y))\big((\alpha_i)_i\big) \arrow[rd, dash] & & \Q(y) \\
& \Q(g(y))\arrow[ru, "\deg g =b - a=p"', dash] &
\end{tikzcd}
\]
Let~$K=\Q(y)$,~$K_g=\Q(g(y))$ and~$c=[K((\alpha_i)_i):K_g(\alpha_i)_i]$. The fact that~$[K:K_g]=\deg g$ comes from~\cite{Cox}*{Prop.~7.5.5, p.~176} since~$g$ is non constant. The extension~$K_g((\alpha_i)_i)/K_g$ is Galois since~$K_g((\alpha_i)_i)$ is a splitting field of~$T_{U_{a,b}}$ over~$K_g$. Let~$H$ denote its Galois group. As easily seen in the above diagram,~$c\cdot |H|=|G|\cdot \deg g$. 

Any automorphism of~$K((\alpha_i)_i)$ fixing~$K$ restricts to an automorphism of~$K_g((\alpha_i)_i)$ fixing~$K_g$ (indeed,~$g$ has rational coefficients). This induces an injective group morphism~$G\hookrightarrow H$. 

Since~$g$ is a non constant polynomial with rational coefficients, one has~$H\simeq \gal_{\Q(Y)}(f(x)+Y)$, where~$Y$ is an indeterminate over~$\Q$ and the Galois group is understood as that of the splitting field of~$f(x)+Y$ over~$\Q(Y)$. We invoke a result of Dujella, Gusi{\'c} and Tichy~\cite{DGT}*{Cor.~1}, that asserts that
the polynomial with integer coefficients
\[
f(x+1)=x^a+bx^{a-1}+\cdots
\]
is indecomposable (meaning that if~$f=f_1\circ f_2$ for some integral polynomials~$f_1$,~$f_2$ then either~$f_1$ or~$f_2$ has degree~$1$; note also that regarding~$f(x+1)$ as an element of~$\Z[x]$ or~$\Q[x]$ does not affect indecomposability, by~\cite{Tur}*{Cor.~2.3}) since, by assumption,~$b$ is coprime to~$a$. In turn the polynomial~$f$ is indecomposable and by virtue of~\cite{Tur}*{Lem.~3.1}, the group~$H$ acts as a primitive permutation group on the roots (seen in a fixed algebraic closure~$\overline{\Q(Y)}$ of~$\Q(Y)$) of the~$x$-polynomial~$f(x)+Y$. In case~$a$ is composite, we can further combine~\cite{Tur}*{Lem.~3.3, Lem.~4.4}. The former statement guarantees that~$H$ contains an~$a$-cycle while the latter infers that in this case~$a$ composite implies that~$H$ is doubly transitive as a permutation group of the roots in~$\overline{\Q(Y)}$ of the~$x$-polynomial~$f(x)+Y$.

\medskip
Finally we prove that~$G\simeq H$. It relies on the following classical property of the compositum of a field extension with a Galois extension (see \emph{e.g.}~\cite{DuFo}*{\S 14.4}).

\begin{lemma}\label{lem:compos}
Let~$\Omega/k$ be a field extension and let~$L$ and~$M$ be subextensions such that~$L/k$ is finite Galois. Then the compositum~$LM$ (inside~$\Omega$) is finite Galois over~$M$ and~$\gal(LM/M)\simeq \gal(L/(L\cap M))$.
\end{lemma}

In order to apply the lemma to our situation, note that~$K((\alpha_i)_i))$ is the compositum of~$K$ and~$K_g((\alpha_i)_i)$ (inside~$\overline{\Q(y)}$). 
The lemma implies that
\[
G=\gal(K\big((\alpha_i)_i\big)/K)\simeq \gal\Big(K_g\big((\alpha_i)_i\big)/(K\cap K_g\big((\alpha_i)_i\big))\Big).
\]
However~$K\cap K_g\big((\alpha_i)_i\big)$ is a subextension of~$K/K_g$ which, by assumption, has prime degree. Therefore~$K\cap K_g\big((\alpha_i)_i\big)$ is either~$K$ or~$K_g$. By contradiction, assume that this intersection is~$K$; then we have a tower~$K_g\big((\alpha_i)_i\big)/K/K_g$ which would imply that~$p\mid [K_g\big((\alpha_i)_i\big):K_g]$. This is impossible by our assumption on~$p$ and since~$[K_g\big((\alpha_i)_i\big):K_g]\leq a!$. We conclude that~$K\cap K_g\big((\alpha_i)_i\big)=K_g$ and in turn that~$G\simeq H$, since~$K_g((\alpha_i)_i)/K_g$ is Galois of group isomorphic to~$H$.

\end{proof}

\begin{remark} \label{rem:bivariate}
Note that~\cite{Tur}*{Lem.~3.3}, used in the proof to derive double transitivity from primitivity in the case where~$a$ is composite, is a property specific to bivariate polynomials (it holds at least for polynomials of the form~$f(x)-y$). In particular, contrary to many other arguments used in the present work, such a strong group-theoretic property on the Galois action is not established by a specialization argument. We will argue in the same way, using an approach specific to bivariate polynomials, in the next section when proving Theorem~\ref{th:CnCase} over number fields. Monodromy groups of type~$\gal_{\C(y)}(f(x)-y)$ where~$f\in \Q[x]$ is indecomposable have been classified~\cite{Mul95}. Also note that similar monodromy computations are performed in~\cite{KaNa}*{\S 7} in the context of curves over finite fields with defining equation of type~$y^a=f(x)$.
\end{remark}

\subsection{The case of an \texorpdfstring{$n$-cycle with a ``thick edge''}{n-cycle with a "thick edge"}}

\begin{definition}\label{def:thick} Let~$n, j \geq 1$ be integers. We define~$C_n^j$ to be the multigraph obtained from a path of length~$n-1$ with endpoints~$u$ and~$v$ by adding~$j$ edges joining~$u$ and~$v$ in parallel.
\end{definition}

Thus~$C_n^{1}=C_n$, the cycle of length~$n$, and~$C_n^j$ has~$n+j-1$ edges and rank~$n-1$. The graph~$C_n^{0}$ is~$P_n$,
the path on~$n$ vertices; the graph~$C_2^0$ consists of a single bridge.

Using the deletion-contraction recurrence for the Tutte polynomial (\emph{e.g.}~\cite{Gordon12}*{Th.~3.1}),
a straightforward induction gives, for~$n \geq 2$, 
\begin{equation}\label{eq:TutteThickCycle}
T(C_n^j;x,y)=x^{n-1} +(y+x+\cdots+x^{n-2})\cdot(1+y+\cdots+y^{j-1}).
\end{equation}
The multigraph~$C_2^j$ consists of~$j+1$ parallel edges joining two vertices, for which 
$T(C_2^j;x, y) = x + y + \cdots + y^j$ ,
consistent with~\eqref{eq:TutteThickCycle} with~$n = 2$. The multigraph~$C_1^j$ consists of~$j$ loops on a vertex, and~$T(C_1^j;x,y) = y^j$.

\subsubsection{The case where~$j$ is odd}
We prove the following extra case of Conjecture~\ref{conj:BCM}. One of the main ingredients here is specializing~$T(C_n^j;x,y)$ at~$y=-1$. Since~$j\equiv 1\bmod 2$ one has for~$n\geq 2$
\begin{equation}\label{eq:Spec-1}
T(C_n^j;x,-1)=x^{n-1}+x^{n-2}+\cdots+x-1.
\end{equation}

\begin{theorem}\label{th:ThickCnCase}
Let~$k/\Q$ be a finite field extension. Assume that~$n\geq 3$ and that~$j$ is odd. We have:
\begin{enumerate}
\item\label{it:1martin} the group~$G_{k,y}(T(C_n^j;x,y))$ is isomorphic to a transitive subgroup of~$\mathfrak S_{n-1}$ that contains a transposition;
\item\label{it:2martin} assuming that~$n$ is odd,
$G_{k,y}(T(C_n^j;x,y))\simeq\mathfrak S_{n-1}$;
\item\label{it:3martin} if~$n-1$ is prime, then~$G_{k,y}(T(C_n^j;x,y))\simeq\mathfrak S_{n-1}$; in fact~$\gal_{\Q}(T(C_n^j;x,-1))\simeq\mathfrak S_{n-1}$.
\end{enumerate}
\end{theorem}

As we will see in the proof,~\eqref{it:3martin} of the statement is deduced from~\cite{Mar04}*{Th.~3.4}. Recall also that Remark~\ref{rem:basefieldext} explains why it is enough to prove the statement in the case~$k=\Q$.
We first state a preparatory result, and proceed by drawing a second proof of Theorem~\ref{th:CnCase} in the
case where~$k$ is a number field.

\begin{lemma}\label{lem:Selmer}
Let~$n\in\N_{\geq 2}$ and let
$
f(x)=x^{n-1}+x^{n-2}+\cdots+x-1
\in\Z[x]$. Let~$G_f$ denote the Galois group of a splitting field of~$f$ over~$\Q$. Then the following hold:
\begin{enumerate}
  \item\label{itselmer1} $f$ is irreducible over~$\Q$; 
  \item\label{itselmer2} if~$n\geq 3$,
then the discriminant of the polynomial~$f$ has an odd prime divisor (in other words,~$|\disc(f)|$ is not a power of~$2$);
\item\label{itselmer3} if~$p$ is an odd prime divisor of~$|\disc(f)|$ then the inertia subgroup at~$p$ (defined up to conjugation and seen as a permutation subgroup of the complex roots of~$f$) is generated by a transposition, and
\item\label{itselmer4} assuming that~$n$ is odd, we have~$G_f\simeq \mathfrak S_{n-1}$.
\end{enumerate}
\end{lemma}

Before proving the Lemma, we use it to give a second proof of Theorem~\ref{th:CnCase} in the case where~$k$ is a number field (which reduces to the case~$k=\Q$ by Remark~\ref{rem:basefieldext}).

\begin{proof}[Alternative proof of Theorem~\ref{th:CnCase} for~$k=\Q$] As mentioned after Definition~\ref{def:thick} the case of an~$n$-cycle corresponds to~$j=1$. Recall also that Corollary~\ref{cor:dMMN} guarantees that~$T_{C_n}$ is irreducible seen as an~$x$-polynomial with coefficients in~$\Q(y)$. Therefore, combining~\eqref{eq:Spec-1} with Proposition~\ref{prop:spec} and Lemma~\ref{lem:Selmer}\eqref{itselmer2} and~\eqref{itselmer3}, the Galois group~$\gal_{\Q(y)}(T_{C_n})$ is isomorphic to a transitive subgroup of~$\mathfrak S_{n-1}$ containing a transposition. Moreover, as mentioned in Remark~\ref{rem:bivariate}, we may combine~\cite{DGT}*{Cor.~1} and~\cite{Tur}*{Lem.~3.1} to show that this Galois group acts as a primitive permutation group of degree~$n-1$. A primitive permutation group of degree~$n-1$ containing a transposition is necessarily isomorphic to~$\mathfrak S_{n-1}$ (see~\cite{DiMo}*{Th.~3.3A}).
\end{proof}

\begin{proof}[Proof of Lemma~\ref{lem:Selmer}]
  For~\eqref{itselmer1}, we first set 
\[
g(x)=(x-1)f(x)=(x-1)\cdot \bigg(\frac{x^n-1}{x-1}-2\bigg)=x^n-2(x-1)-1=x^n-2x+1.
\]
We then apply a result due to Perron (see \emph{e.g.}~\cite{Sel}*{Th.~2}) asserting that~$f(x)$ (\emph{i.e.} the quotient of~$g(x)$ by~$x-1$) is irreducible over~$\Q$.

\medskip
Next we prove~\eqref{itselmer2}. Fix~$n\geq 3$.
The discriminant~$D$ of~$g$ is given by~(\cite{Swa}*{Th.~2}):
\begin{equation}\label{eq:discg}
D=(-1)^{n(n-1)/2}\big(n^n-(n-1)^{n-1}2^n\big).
\end{equation}
Moreover, recalling the definition of the discriminant of a polynomial in terms of the resultant of the polynomial and its derivative (\cite{Lang}*{Chap.~IV, Prop.~8.3 and~8.5}), we have:
\begin{align}\label{eq:SquareQuotient}
D=(-1)^{\frac{n(n-1)}2}\res(g,g')&=(-1)^{\frac{n(n-1)}2}\prod_{\alpha\text{ root of }g}
g'(\alpha)\\ \nonumber
&= (-1)^{\frac{n(n-1)}2}f(1)\prod_{\alpha\text{ root of }f}
f'(\alpha)\prod_{\alpha\text{ root of }f}(\alpha-1)\\ \nonumber
&=(-1)^{\deg f+\frac{n(n-1)}2}(f(1))^2\prod_{\alpha\text{ root of }f}
f'(\alpha)=(f(1))^2\mathrm{disc}(f).
\end{align}
Since~$f(1)=n-2$, we obtain:
\[
|\mathrm{disc}(f)|=\frac{|n^n-(n-1)^{n-1}2^n|}{(n-2)^2}.
\]

If~$n$ is odd then the right hand side is odd as well and we are done\footnote{Note that~$\disc(f)\neq\pm 1$ because~$f$ is irreducible and has degree~$>1$ and thus generates a ramified extension of the rationals.}. Let us assume that~$n=2k'$ ($k'\geq 2$). Then, up to sign, we have
\[
|\disc(f)|=\frac{|(2k')^{2k'}-(2k'-1)^{2k'-1}2^{2k'}|}{4(k'-1)^2}
=2^{2(k'-1)}\frac{|k'^{2k'}-(2k'-1)^{2k'-1}|}{(k'-1)^2}
\]
If~$k'$ is even, then the right-most factor is odd\footnote{Here we note that the numerator of the fraction of the right-most member is~$\geq k'^{2k'-1}$ for~$k'\geq 2$ and therefore the fraction cannot be~$1$.} and we are done. Therefore we assume that~$n=2(2k+1)$ (\emph{i.e.}~$k'=2k+1$) for some~$k\geq 1$. We get
\[|\disc(f)|=2^{4k}\frac{|(2k+1)^{2(2k+1)}-(4k+1)^{4k+1}|}{4k^2}\]
Let us investigate the parity of the second factor.
By Newton's formula:
\[
(2k+1)^{2(2k+1)}=\sum_{j=0}^{4k+2}\binom{4k+2}{j}(2k)^j,\qquad
(4k+1)^{4k+1}=\sum_{i=0}^{4k+1}\binom{4k+1}{i}(4k)^i.
\]
The contributions of the indices~$i=j=0$ are both~$1$. They subtract to~$0$. Moreover we see that the contributions of the indices:
\begin{enumerate}
\item[(i)]~$j=1$ and~$i=1$ are respectively~$(2k+1)4k$ and~$(4k+1)4k$. The difference subtracts to an integer divisible by~$8k^2$.
\item[(ii)]~$j=2$ and~$i=2$ are respectively~$4k^2(2k+1)(4k+1)$ and~$16k^2 2k(4k+1)$ and the difference of these two terms is divisible par~$4k^2=2^{2\ell+2}$ but not~$8k^2$ since~$(2k+1)(4k+1)$ is odd.
\item[(iii)]~$j\geq 3$ and~$i\geq 3$ are integers divisible by~$8k^2$.
\end{enumerate}
We conclude that 
\begin{equation}\label{eq:disccompute}
\frac{|(2k+1)^{2(2k+1)}-(4k+1)^{4k+1}|}{4k^2}\equiv |(4k+1)(2k+1-8k)|\equiv 1\bmod 2.
\end{equation}
Therefore,~$|\disc(f)|$ admits an odd prime divisor as soon as the left-hand side of~\eqref{eq:disccompute} is not equal to~$1$.
We claim that indeed one has
\[
(4k+1)^{4k+1}-(2k+1)^{2(2k+1)} > 4k^2.
\]
Notice first that the inequality is true if~$k\in\{1,2\}$ by direct computation and suppose
now that~$k\ge3$. It is enough to prove that
\[
    \frac{(4k+1)^{4k+1}}{(2k+1)^{4k+2}} > (4k+1)^2,
\]
that is
\[
    {\left(\frac{4k+1}{2k+1}\right)}^{4k-1} > (2k+1)^3.
\]
Taking the logarithm on both sides, we see that it suffices that
\[
    (4k-1)\ln(13/7) > 3\ln(2k+1),
\]
because~$x\mapsto\frac{4x+1}{2x+1}=1+\frac{2x}{2x+1}$ is increasing as a function of~$x$
and we assumed that~$k\ge3$.
This last inequality holds for all~$k\ge3$ as can be seen by a quick analysis
of the function~$x\mapsto (4x-1)\ln(13/7)-3\ln(2x+1)$ for~$x>1$.
We have proved that~$|\disc(f)|$ is not a power of~$2$.

\medskip
We turn to~\eqref{itselmer3}.
Our method mimics the argument of Osada~\cite{Osa}*{Proof of Th.~1}. 

Let~$K$ be the splitting field (inside the complex numbers) of~$f$ (equivalently, of~$g$) over~$\Q$. Let~$p$ be a prime number ramified in~$K/\Q$ and let~$\mathfrak p$ be a prime ideal of~$\mathcal O_K$ lying over~$p$. We let~$I(\mathfrak p\slash p)$ denote the inertia subgroup of~$G_f$ relative to~$p$ and~$\mathfrak p$. Let~$\sigma$ be a non trivial element of~$I(\mathfrak p\slash p)$; thus~$\sigma(\alpha)\neq \alpha$ for some root~$\alpha$ of~$f$. Since~$\sigma\in I(\mathfrak p\slash p)$, we have~$\sigma(\alpha)\equiv\alpha (\bmod\ \mathfrak p)$, meaning that~$\alpha \bmod\ \mathfrak p$ is a multiple root of the reduction of~$f$ modulo~$p$. In particular~$g (\bmod\ p)\in \F_p[x]$ has a multiple root. Let us show that~$g (\bmod\ p)$ has at most one multiple root, and that in this case, its multiplicity is~$2$. We will deduce that for any root~$\beta\in K$ of~$f$, the necessary congruence~$\sigma(\beta)\equiv \beta (\bmod\ \mathfrak p)$ implies~$\sigma(\beta)=\beta$ and thus~$\sigma$, seen as a permutation of the roots of~$f$ is the transposition,~$(\alpha \sigma(\alpha))$.

Let~$r$ be a multiple root of~$f(x)$ modulo some odd prime~$p$ ramified in~$K/\Q$. It is also a multiple root of~$g(x)=x^n-2x+1$, seen as an element of~$\F_p[x]$. Then~$r^n=2r-1$ and~$r$ is also a root of the derivative~$nx^{n-1}-2$ of~$g$. Thus~$nr^n=2r$. In particular~$p\nmid n$: indeed, by contradiction, if~$p\mid n$, then~$nr^n=2r=0$ which would imply~$p=2$ since~$r\neq 0$. This contradicts the fact that~$p$ is odd. Hence~$2r-1=2r/n$ and in turn~$2r(1-1/n)=1$. In particular~$p\nmid n-1$. We conclude that~$r=n/(2(n-1))$ is the only possible multiple root of~$x^n-2x+1\in \F_p[x]$. Finally the second derivative of~$g$ is~$n(n-1)x^{n-2}$ which only vanishes at~$0$ (recall~$p\nmid n$ and~$p\nmid n-1$). Thus~$r$ has mulitplicity~$<3$ as a root of~$g$.

\medskip
Finally we prove~\eqref{itselmer4}.  Since~$n\geq 2$ and~$n$ is odd, we have~$n\geq 3$. First remark that the Galois group~$G_f$ of~$f$ over~$\Q$ is the same as the Galois group of~$g$ over~$\Q$. By~\eqref{itselmer1}, the group~$G_f$ is a transitive subgroup of~$\mathfrak S_{n-1}$. By~\cite{Osa}*{Lemma 5}, it suffices to show that~$G_f$ (seen as a subgroup of~$\mathfrak S_{n-1}$) can be generated by transpositions to conclude that~$G_f\simeq \mathfrak S_{n-1}$. To do so, we use the fact that~$G_f$ is generated by the union over prime numbers~$p$ of the inertia subgroups of~$G_f$ at~$p$ (a consequence of Galois theory combined with the fact that no non-trivial extension of~$k=\Q$ is unramified). 

Note that~$f$ does not have a multiple root modulo~$2$. Indeed, assume by contradiction that~$r$ is such a root; it is a multiple root of~$g$ and by the same computation as the one performed in the proof of~\eqref{itselmer3}, we have:
\[
r^n=2r-1=1\,,\qquad nr^{n-1}=0.
\]
These two equalities are not compatible, since~$n$ is odd. This implies that all the ramified primes in the splitting field of~$f$ over~$\Q$ are odd and by~\eqref{itselmer3} all the non trivial inertia subgroups of~$G_f$ are generated by a transposition. This concludes the proof.
\end{proof}

We are now ready to prove Theorem~\ref{th:ThickCnCase}; our argument combines Lemma~\ref{lem:Selmer} and Proposition~\ref{prop:spec}.

\begin{proof}[Proof of Theorem~\ref{th:ThickCnCase}]
We first prove~\eqref{it:1martin}. Recall~\eqref{eq:Spec-1} and apply Proposition~\ref{prop:spec} for the choice~$R=k[y]$, ~$f(x)=T(C_n^j;x,y)$ (irreducible over~$R$ by Corollary~\ref{cor:dMMN}, and monic as a polynomial with coefficients in~$k[y]$) and where the ring homomorphism we choose is reduction modulo~$y+1$. The image of the polynomial~$T(C_n^j;x,y)$ by such a morphism is the~$\Q$-polynomial on the right hand side of~\eqref{eq:Spec-1}.
We conclude by combining Lemma~\ref{lem:Selmer}\eqref{itselmer1},~\eqref{itselmer2} and~\eqref{itselmer3}.

The proof of~\eqref{it:2martin} follows from the same specialization argument (Proposition~\ref{prop:spec}) as in the proof of~\eqref{it:1martin},
combined this time with Lemma~\ref{lem:Selmer}\eqref{itselmer4}.

Finally we prove~\eqref{it:3martin}. 
This is deduced from~\cite{Mar04}*{Th.~3.4} which asserts that the polynomial~$T(C_{p+1}^j;x,-1)$ is irreducible modulo~$p\coloneqq n-1$. Let~$h(x)=x^p-\sum_{i=0}^{p-1}x^i$. Note that for odd~$j$, this is the opposite of the reciprocal of~$T(C_{p+1}^j;x,-1)$ (see~\eqref{eq:Spec-1}). We further denote by
~$h_p(x)=x^p-\sum_{i=0}^{p-1}x^i\in\F_p[x]$ the reduction of~$h$ modulo~$p$. It is enough to prove that~$h_p$ is irreducible to deduce the result. Indeed by virtue of~\eqref{eq:discg} and~\eqref{eq:SquareQuotient} the prime~$p$ does not divide the discriminant of~$f$ (recall that a polynomial with non zero constant coefficient and its reciprocal have, up to sign and a power of the constant coefficient, the same discriminant), therefore if~$h_p$ is irreducible, the Galois group of~$h$ over~$\Q$ contains a~$p$-cycle and is a subgroup of~$\mathfrak S_p$. It is therefore a transitive permutation group of prime degree, hence it is a primitive permutation group. Since it also contains a transposition (by Lemma~\ref{lem:Selmer}\eqref{itselmer2} and~\eqref{itselmer3}), we conclude that~$\gal_\Q(h)\simeq \mathfrak S_p$ by~\cite{DiMo}*{Th.~3.3A}. Finally we apply once more Proposition~\ref{prop:spec}.
\end{proof}

\subsubsection{Case~$n-1\in\{p,p^2\}$ and~$-j$ nonsquare mod~$p$}
We prove that Conjecture~\ref{conj:BCM} holds for~$T(C_n^j;x,y)$ for extra values of~$n$ and~$j$. This time we proceed by specializing the Tutte polynomial at~$y=1$:
\[
T(C_n^j;x,1)\eqqcolon h(x)=x^{n-1}+j\sum_{k=0}^{n-2}x^k.
\]

\begin{theorem}\label{th:ThickCnCase-2}
Let~$k/\Q$ be a finite field extension. If
\begin{enumerate}
\item~$n-1\in\{p,p^2\}$ for some prime number~$p\geq 5$, \item~$-j$ is a nonsquare modulo~$p$, and
\item~$\gcd(n,j-1)=1$,
\end{enumerate}
then~$G_{k,y}(T(C_n^j;x,y))\simeq\mathfrak S_{n-1}$.
\end{theorem} 

As before, it is enough, by Remark~\ref{rem:basefieldext} to consider the case~$k=\Q$. We will need the following preparatory result.

\begin{lemma}\label{lem:sqdisc}
Let~$a,b\in\Q$ and let~$m>k>0$ be integers such that~$\gcd(m,k)=1$. Consider the trinomial~$f(x)=x^m+ax^k+b\in\Q[x]$. Let~$h(x)\in\Z[x]$ be an irreducible factor of~$f(x)$ and assume that
\begin{enumerate}
\item[(i)]~$D_f/D_h$ is a square in~$\Z$ (here~$D_g$ denotes the discriminant of any~$g\in\Q[x]$);
\item[(ii)] there exists a prime number~$p$ such that~$v_p(D_f)$ (the~$p$-adic valuation of~$D_f$) is odd and~$p\nmid ab$. 
\end{enumerate}
Then the Galois group~$\gal_{\Q}(f)$ of the splitting field of~$f$ (seen as a permutation subgroup of the complex roots of~$f$) over~$\Q$ contains a transposition.
\end{lemma}

\begin{proof}
We use~\cite{Swa}*{Th.~2} to compute:
\[
D_f=(-1)^{m(m-1)/2}b^{k-1}\big(m^mb^{m-k}+(-1)^{m+1}(m-k)^{m-k}k^ka^m\big).
\]

Let~$\theta$ be a complex root of~$h$ and denote by~$d_{\Q(\theta)}$ the discriminant of the number field~$\Q(\theta)$, then~$D_h/d_{\Q(\theta)}$ is a square in~$\Z$ (the square of the index of~$\Z[\theta]$ in the ring of integers of~$K$). By assumption~(i), we deduce that~$D_f/d_{\Q(\theta)}=(D_f/D_h)\cdot (D_h/d_{\Q(\theta)})$ is a square in~$\Z$.

We deduce that if~$p$ is a prime number satisfying (ii) then~$p$ is ramified in~$\Q(\theta)/\Q$ (\emph{i.e.} it divides~$d_{\Q(\theta)}$). Moreover, applying~\cite{MoSa94}*{Lemma 5} (which builds on~\cite{LNV84}*{Th.~2}), the inertia subgroup relative to any prime ideal~$\mathfrak p$ above~$p$ in the splitting field~$K_h$ of~$h$ over~$\Q$ is generated by a transposition. In turn~$\gal_{\Q}(f)$ contains a transposition.
\end{proof}

Our goal is to apply the above Lemma to the specialization~$h(x)$ at~$y=1$ of~$T(C_n^j;x,y)$. Then~$f(x)=(x-1)h(x)$ is a trinomial to which the Lemma can be applied if 
\begin{enumerate}
\item[(a)]~$h$ is irreducible,
\item[(b)] one can find a prime~$p$ such that~$v_p(D_f)$ is odd and~$p\nmid j(j-1)$ (see the proof of Theorem~\ref{th:ThickCnCase-2} below).
\end{enumerate}

To see that condition (a) holds we invoke~\cite{Har}*{Th.~2} asserting the irreducibility in~$\Q[x]$ of~$T(C_n^j;x,1)$ for~$j>1$.

\begin{lemma}[Harrington]\label{lem:irredat1}
Assume~$n\geq 2$. The specialization at~$y=1$ of~$T(C_n^j;x,y)$ is the~$\Q$-polynomial:
\[
h(x)=x^{n-1}+j\sum_{k=0}^{n-2}x^k.
\]
If~$j>1$, then~$h(x)$ is irreducible over~$\Q$ except if~$(n,j)=(3,4)$.
\end{lemma}

\begin{proof}[Proof of Theorem~\ref{th:ThickCnCase-2}]
Recall that 
$G=G_{\Q,y}(T(C_n^j;x,y))$ seen as a permutation group is a transitive subgroup of~$\mathfrak S_{n-1}$ by virtue of Corollary~\ref{cor:dMMN}.

Next let us show that~$G$ contains a transposition. 
Let~$h(x)$ be the specialization of~$T(C_n^j;x,y)$ at~$y=1$ and define~$f(x)\coloneqq (x-1)h(x)=x^n+(j-1)x^{n-1}-j$. Then, as shown in the proof of~\ref{lem:Selmer}\eqref{itselmer2}, the number~$D_f/D_h$ is a square in~$\Z$. Furthermore, as seen in the proof of Lemma~\ref{lem:sqdisc}, one has 
\[
D_f=(-1)^{n(n-1)/2}(-j)^{n-1}D_0 \text{ where } D_0=(-j)n^n+(-1)^{n-1}(n-1)^{n-1}(j-1)^n.
\]
We note that~$D_0 \equiv -j\bmod (n-1)$, therefore~$D_0$ is a non-square modulo~$p$. 
In particular~$D_0$ is not a square in~$\Z$. Let~$\ell$ be a prime such that~$v_\ell(D_0)$ is odd. Then~$\ell\nmid j-1$, otherwise~$\ell$ would divide either~$n$ or~$j$ contradicting the assumption~$\gcd(n,j-1)=1$.

Also note that~$j$ is coprime to~$D_0$ (in particular~$\ell\nmid j$). Indeed a common prime divisor of~$j$ and~$D_0$ would divide~$n-1$ and so that prime divisor would be~$p$. We would obtain~$j\equiv 0 \bmod p$ which contradicts the fact that~$-j$ is a nonsquare modulo~$p$. Therefore~$D$ is a non-square in~$\Z$.

Therefore since~$v_\ell(D_f)=v_\ell(D_0)$ is odd and~$\ell\nmid j(j-1)$ Lemma~\ref{lem:sqdisc} applies and shows that~$\gal_\Q(f)$ seen as a permutation group of the complex roots of~$f$ contains a transposition. In particular~$G$ contains a transposition.

Finally we want to prove that~$G$ is a primitive subgroup of~$\mathfrak S_{n-1}$, which is now enough to conclude by~\cite{DiMo}*{Th.~3.3A}. The proof of primitivity is provided by Proposition~\ref{prop:MoSa-gen} below.

In order to see that the assumptions of Proposition~\ref{prop:MoSa-gen} are satisfied for~$h(x)=T(C_n^j;x,1)$, we appeal to Lemma~\ref{lem:irredat1} and we set~$\phi(x)\coloneqq (x-1)h(x)=x^n+(j-1)x^{n-1}-j$ and so, in the notation of Proposition~\ref{prop:MoSa-gen}, one has~$k=n$,~$s=n-1$,~$a=j-1$ and~$b=-j$.
One cheks that~$\gcd((j-1)(n-1),-jn)=1$ since~$(n-1)$ (resp.~$j-1$) is coprime to both~$n$ and~$-j$. 
\end{proof}

\begin{proposition}\label{prop:MoSa-gen}
Let~$\phi(x)=x^k+ax^{s}+b\in\Z[x]$ for integers~$k\geq 3$ and~$s\in\{1,\ldots,k-1\}$.

Assume that~$\mathrm{gcd}(as(k-s),kb)=1$ and~$s\in \{p,p^2\}$ for some prime number~$p$. If~$\phi(x)=(x-\alpha)h(x)$ for some~$\alpha\in\Z$ and an irreducible~$h\in \Z[x]$, then the Galois group of the splitting field of~$\phi$ over~$\Q$ acts as a primitive permutation group on the set of complex roots of~$h$.
\end{proposition}

\begin{proof}
The setting and the proof are slightly adapted from work~\cite{MoSa94} of Movahhedi--Salinier. Letting~$G_\phi$ be the Galois group of the splitting inside~$\C$ of~$\phi$ over~$\Q$, the assumptions imply that~$G_\phi$ is isomorphic to a transitive subgroup of~$\mathfrak S_{k-1}$. Moreover using a computation similar to~\eqref{eq:SquareQuotient} one has
\[
\mathrm{disc}(\phi)=h(\alpha)^2\mathrm{disc}( h)
\]
and thus the analysis of~\cite{MoSa94}*{\S 2, \S3} applies and we conclude by invoking~\cite{MoSa94}*{Cor.~1 to Th.~3}.
\end{proof}

\section*{Acknowledgements}

F. Jouve benefited from the financial support of the ANR through project ETIENE (ANR-24-CE93-0016) and is grateful to its members for valuable remarks and comments after some of the above results were presented on the occasion of the first project meeting held in Jussieu in June 2025.
A. Goodall was partially supported by the Czech Science Foundation (GA\v{C}R) grant 25-16627S.

\end{document}